\documentclass[10pt,reqno]{amsart}
\usepackage{amsmath}
\usepackage{amsthm}
\usepackage{amssymb}
\usepackage{amsfonts}
\usepackage{subcaption}
\usepackage{graphicx}
\usepackage{latexsym}
\usepackage{cite, url, xcolor}
\usepackage{hyperref}
\usepackage{enumitem}
    \setenumerate{itemsep=5pt} 
    \setitemize{itemsep=5pt} 

\usepackage{mathtools}    
\mathtoolsset{showonlyrefs}

    \usepackage{tikz} 
    \usetikzlibrary{decorations.pathreplacing}
    \usetikzlibrary{arrows,shapes,backgrounds,3d}
\usetikzlibrary{decorations.pathreplacing,shapes,snakes,automata}
\usepackage{tkz-graph}

\usepackage[font=small,labelfont=bf]{caption}

\usepackage[sc]{mathpazo}
\usepackage[T1]{fontenc} 
\newcommand{\ket}[1]{\mathbf{#1}}

\renewcommand{\vec}[1]{\mathbf{#1}}

\renewcommand{\P}{\mathcal{P}}
\newcommand{\J}{\mathcal{J}}
\newcommand{\K}{\mathcal{K}}

\renewcommand{\S}{\mathfrak{S}}
\renewcommand{\phi}{\varphi}

\newcommand{\R}{\mathbb{R}}
\newcommand{\N}{\mathbb{N}}

\newcommand{\G}{\mathcal{G}}

\newcommand{\M}{\mathrm{M}}

\newcommand{\C}{\mathcal{C}}

\newcommand{\V}{\mathcal{V}}

\newcommand{\tr}{\operatorname{tr}}

\newcommand{\inner}[1]{\langle #1 \rangle}

\newcommand{\minimatrix}[4]{\begin{bmatrix} #1 & #2\\#3 & #4 \end{bmatrix}}

\newcommand{\Orb}{\operatorname{Orb}}

\newcommand{\0}{{\color{lightgray}0}}

 \numberwithin{equation}{section}
\newtheorem{theorem}[equation]{Theorem}
\newtheorem{lemma}[equation]{Lemma}
\newtheorem{proposition}[equation]{Proposition}
\newtheorem{corollary}[equation]{Corollary}
\theoremstyle{definition}

\newtheorem{example}[equation]{Example}

\newtheorem{problem}[equation]{Problem}

\allowdisplaybreaks
\begin{document}
\title{Symmetric tensor powers of graphs}

\begin{abstract}
The symmetric tensor power of graphs is introduced and its fundamental properties are explored.  A wide range of
 intriguing phenomena occur when one considers symmetric tensor powers of familiar graphs. 
A  host of open questions are presented, hoping to spur future research.
\end{abstract}

\author[Astaiza]{Weymar Astaiza}
\address{Departamento de Matem\'aticas, Universidad del Cauca, Popay\'an, Colombia}
\email{wastaiza@unicauca.edu.co}

\author[Barrios]{Alexander J. Barrios}
\address{Department of Mathematics
University of St. Thomas
2115 Summit Avenue
St. Paul, MN 55105}
\email{abarrios@stthomas.edu}

\author[Chimal-Dzul]{Henry Chimal-Dzul}
\address{Department of Mathematics, University of Notre Dame}
\email{hchimald@nd.edu}

\author[Garcia]{Stephan Ramon Garcia}
\address{Department of Mathematics and Statistics, Pomona College, 610 N. College Ave., Claremont, CA 91711 USA} 
\email{stephan.garcia@pomona.edu}
\urladdr{\url{http://pages.pomona.edu/~sg064747}}

\author[de la Luz]{Jaaziel Lopez de la Luz}
\address{Department of Mathematics, UC Irvine}
\email{jaaziel@uci.edu}

\author[Moll]{Victor H. Moll}
\address{Department of Mathematics, Tulane University, New Orleans, LA 70118 USA}
\email{vhm@tulane.edu}

\author[Puig]{Yunied Puig}
\address{Lafayette College}
\email{puigdedy@lafayette.edu}

\author[Villamizar]{Diego Villamizar}
\address{Escuela de Ciencias Exactas e Ingenier\'ia\\Universidad Sergio Arboleda\\Bogot\'a, Colombia}
\email{diego.villamizarr@usa.edu.co}
\urladdr{\url{https://sites.google.com/view/dvillami/}}

\thanks{SRG acknowledges the support of National Science Foundation (NSF) grant DMS-2054002.  We also gratefully acknowledge the support
of the American Institute of Mathematics (AIM)}


\maketitle



\section{Introduction}\label{Section:Introduction}

A great variety of graph products, and hence graph powers, exist in the literature \cite{hammack-2011a}.
For example, the tensor (or Kronecker) product of graphs has a long and fruitful history \cite{henderson-1983a}.
In this paper we introduce the symmetric tensor power of graphs, a graph power that displays a variety of intriguing
phenomena.  Although the resulting graphs often bear surprising and counterintuitive features, the approach is well-motivated algebraically.
Indeed, the naturalness of the symmetric tensor power is illustrated by its compatibility with graph spectra.


The structure of this paper is as follows.
Section \ref{Section:Linear} covers some linear-algebraic preliminaries before
symmetric tensor powers of graphs are defined in Section \ref{Section:Products} (with some computational material deferred until 
Appendix \ref{Section:WellDefined}).
The connection between this operation and graph spectra is considered in Section \ref{Section:Spectrum}.
We make several combinatorial observations in Section \ref{Section:Combinatorial} 
and study in Section \ref{Section:Particular} several curious phenomena that arise for familiar graphs.
Section \ref{Section:Wiener} concerns the Wiener index of certain graphs.
We conclude with a host of open questions about symmetric tensor powers of graphs, which indicates that
the subject is fertile ground for future exploration, in Section \ref{Section:Open}.

\section{Linear-algebraic preliminaries}\label{Section:Linear}

In what follows, we denote by $\M_n$ the set of real $n \times n$ matrices,
$\N := \{1,2,\ldots\}$ the set of natural numbers,
$[n] := \{1,2,\ldots,n \}$, and $|X|$ the cardinality of a set $X$.

Let $\V$ denote a real inner-product space with orthonormal basis $v_1,v_2,\ldots,v_n$.
For $k \in \N$, the $k$th \emph{tensor power} of $\V$ is the $n^k$-dimensional inner-product space $\V^{\otimes k}$ 
spanned by the \emph{simple tensors}
$v_{i_1} \otimes v_{i_2} \otimes  \cdots\otimes  v_{i_k}$, in which $(i_1,  i_2, \ldots, i_{k} ) \in [n]^k$,
and endowed with the inner product that linearly extends
\begin{equation}\label{eq:InnerProduct}
	\inner{ v_{i_1} \otimes v_{i_2} \otimes \cdots v_{i_k}, \,
	v_{j_1} \otimes v_{j_2} \otimes \cdots v_{j_k}}
	:=
	\inner{ v_{i_1}, v_{j_1}}\inner{ v_{i_2}, v_{j_2}}\cdots \inner{ v_{i_k}, v_{j_k}}.
\end{equation}
In particular, the $v_{i_1} \otimes v_{i_2} \otimes  \cdots\otimes  v_{i_k}$ comprise an orthonormal basis of $\V^{\otimes k}$.

The symmetric group $\S_{k}$ acts on $[n]^{k}$ by permutation. 
Let $\Orb(\vec{i})$ denote the orbit of $\vec{i} = (i_1,i_2,\ldots,i_k)$ under this action; that is,  
$\vec{j} \in \Orb(\vec{i})$ if and only if $\vec{j}$ is a permutation of $\vec{i}$. 
We let $\S_k$ permute simple tensors in the analogous manner.

Let $\V^{\odot{k}}$ denote the subspace of $\V^{\otimes k}$ spanned by the \emph{symmetric tensors}
\begin{equation}\label{eq:SymmetricTensor}
    v_{i_1} \odot v_{i_2} \odot \cdots \odot v_{i_k}
    := \frac{1}{k!} \sum_{\sigma \in \S_k} v_{i_{\sigma(1)}} \otimes v_{i_{\sigma(2)}} \otimes \cdots \otimes  v_{i_{\sigma(k)}}.
\end{equation}
Symmetric tensors are invariant under the action of $\S_k$, 
so there is a representative $k$-tuple for \eqref{eq:SymmetricTensor}
such that $i_1 \leq i_2 \leq \cdots \leq i_k$.
We have $\dim \V^{\odot{k}}  = \binom{n +k-1}{k}$.

\begin{example}\label{Example:Basic}
Let $n=k=2$.
Then $\V$ has orthonormal basis $v_1, v_2$ and $\V^{\otimes 2}$ is $4$-dimensional with
orthonormal basis $v_1 \otimes v_1 , \, v_1  \otimes v_2 , \, v_2  \otimes v_1 , \, v_2  \otimes v_2  $.
The $\S_2$-orbits in $\V^{\otimes 2}$ are 
$\Orb( v_1  \otimes v_1  )  =  \{ v_1  \otimes v_1  \}$,
$\Orb( v_1  \otimes v_2  )  =  \{ v_1  \otimes v_2 , \, v_2  \otimes v_1  \}$, and
$\Orb( v_2  \otimes v_2  )  =  \{ v_2  \otimes v_2  \}$. Thus,
$v_1  \otimes v_1 $,
$\tfrac{1}{2} ( v_1  \otimes v_2  + v_2  \otimes v_1  )$,
$v_2  \odot v_2 $ comprise a basis for $\V^{\odot{2}}$, which is $\binom{2+2-1}{2} = \binom{3}{2} = 3$ dimensional.
\end{example}

For $\vec{i} = (i_1,i_2,\ldots,i_k) \in [n]^k$, let
$\vec{m}(\vec{i})= (m_1,m_2,\ldots,m_n)$, in which each $m_{\ell}$ is the number of occurrences of $\ell$ in $\vec{i}$.
For example, if $\vec{i} = (1,3,2,4,3,1) \in [5]^6$, then $\vec{m}(\vec{i})= (2,1,2,1,0)$. 
We may write $\vec{m}$ if the dependence on $\vec{i}$ is clear. There are
\begin{equation*}
\binom{k}{\vec{m}} := \frac{k!}{m_1! m_2! \cdots m_n!}
\end{equation*}
elements of $[n]^k$ that give rise to the same symmetric tensor
$v_{i_1} \odot v_{i_2} \odot \cdots \odot v_{i_k}$.  The quantity above equals the cardinality of  $\Orb(\vec{i})$. 

\begin{lemma}\label{Lemma:ONB}
Fix $n,k \in \N$ and let $N = \binom{n +k-1}{k}$. The $N$ vectors
\begin{equation}\label{eq:ONB}
u_{\vec{i}} :=  \binom{k}{ \vec{m}(\vec{i})}^{1/2}
v_{i_1} \odot v_{i_2} \cdot \cdots \odot v_{i_k},
\end{equation}
where $\vec{i} = (i_1,i_2,\ldots,i_k) \in [n]^k$ is nondecreasing, form
an orthonormal basis for $\V^{\odot k}$. 
\end{lemma}

\begin{proof}
If $\vec{i},\vec{j} \in [n]^k$ are nondecreasing and distinct, 
they are in different $\S_k$ orbits in $[n]^k$, so \eqref{eq:InnerProduct} ensures that
$u_{\vec{i}}$ and $u_{\vec{j}}$ are orthogonal.
Since $\dim \V^{\odot{k}}  = N$, we need only
show that each $u_{\vec{i}}$ is a unit vector.  To this end, observe that
for each of the $N$ nondecreasing $\vec{i} = (i_1,i_2,\ldots,i_k) \in [n]^k$, 
\begin{align*}
\bigg\langle  \bigodot_{\ell=1}^k \ket{i}_{\ell}, \,\bigodot_{\ell=1}^k \ket{i}_{\ell} \bigg \rangle
&=\bigg\langle
\frac{1}{k!} \sum_{\sigma \in \S_k} \bigotimes_{\ell=1}^k \ket{i}_{\sigma(\ell)} , \,
\frac{1}{k!} \sum_{\tau \in \S_k} \bigotimes_{\ell=1}^k \ket{i}_{\tau(\ell)} 
\bigg\rangle \\
&=\frac{1}{(k!)^2} \sum_{\sigma \in \S_k} \bigg( \sum_{\tau \in \S_k} \prod_{\ell=1}^k \big\langle \ket{i}_{\sigma(\ell)} , \,\ket{i}_{\tau(\ell)}  \big\rangle \bigg) \\
&=\frac{1}{(k!)^2} \sum_{\sigma \in \S_k} \big| \big\{ \tau \in \S_k : \text{$\ket{i}_{\tau(\ell)} = \ket{i}_{\sigma(\ell)}$ for all $\ell \in [n]$} \big\} \big|\\
&=\frac{1}{(k!)^2} \!\!\sum_{\sigma \in \S_k} m_1! m_2! \cdots m_n! 
=\frac{m_1! m_2! \cdots m_n!}{k!} 
= \binom{k}{\vec{m}(\vec{i})}^{-1}\!\!\!\!. \qedhere
\end{align*}
\noindent
\end{proof}

Let $A$ be a linear operator on $\V$.
For $i,j \in [n]$, the $(i,j)$ matrix entry of $A$ with respect to the orthonormal basis $v_1 , v_2 ,\ldots, v_n$ is
$[A]_{i,j} = \inner{ A v_j, v_i}$.  Define
\begin{equation}\label{eq:SymADefi}
A^{\odot k}(v_{i_1} \odot v_{i_2} \odot \cdots \odot v_{i_k})
:= \frac{1}{k!}\sum_{\sigma \in \S_k} (Av_{i_{\sigma(1)}}) \otimes (Av_{i_{\sigma(2)}}) \otimes \cdots \otimes (Av_{i_{\sigma(k)}})
\end{equation}
and extend this by linearity to $\V^{\odot k}$.  The right side above is permutation invariant, so 
$A^{\odot k}$ is a linear transformation from $\V^{\odot k}$ to itself.

\begin{lemma}\label{Lemma:MatrixRep}
The matrix entries of $A^{\odot k}$ with respect to the orthonormal basis \eqref{eq:ONB} are
\begin{equation*}
[A^{\odot k}]_{\vec{i}, \vec{j}}
= \frac{1}{ \sqrt{ \binom{k}{\vec{m}(\vec{i})} \binom{k}{\vec{m}(\vec{j})}} }
    \sum_{\substack{ \vec{p} \in \Orb( \vec{i}) \\ \vec{q} \in \Orb( \vec{j}) }} 
        \prod_{\ell=1}^{k} [A]_{p_{\ell}, q_{\ell}} ,
\end{equation*}
in which $\vec{i}, \vec{j}$ run over the $N = \binom{n +k-1}{k}$ nondecreasing
elements of $[n]^k$ and $\vec{p} = (p_1,p_2,\ldots,p_k), \, \vec{q} = (q_1,q_2,\ldots,q_k)$
belong to the $\S_k$ orbits of $\vec{i}$ and $\vec{j}$ in $[n]^k$, respectively.  In particular, 
$A^{\odot k}$ is real symmetric if $A$ is real symmetric.
\end{lemma}

\begin{proof}
A computation  using stabilizers gives 
\begin{align*}
[A^{\odot k}]_{\vec{i}, \vec{j}}
&= \big \langle A^{\odot k} u_\vec{j}, u_{\vec{i} } \big \rangle  \\
&= \sqrt{ \binom{k}{\vec{m}(\vec{i})} \binom{k}{\vec{m}(\vec{j})}} 
        \left \langle A^{\odot k} (v_{j_1} \odot v_{j_2} \cdot \cdots \odot v_{j_k}), \,
        v_{i_1} \odot v_{i_2} \cdot \cdots \odot v_{i_k} \right\rangle \\
&=   \sqrt{ \binom{k}{\vec{m}(\vec{i})} \binom{k}{\vec{m}(\vec{j})}}  \sum_{\sigma,\tau \in \S_k} 
        \bigg \langle \bigotimes_{\ell=1}^{k} A v_{j_{\sigma(\ell)}}  ,\, 
        \bigotimes_{\ell=1}^{k} v_{i_{\tau(\ell)}} \bigg \rangle \\
&=  \sqrt{ \binom{k}{\vec{m}(\vec{i})} \binom{k}{\vec{m}(\vec{j})}} 
    \sum_{\sigma,\tau \in \S_k} 
        \prod_{\ell=1}^{k} \langle Av_{j_{\sigma(\ell)}} , \, v_{i_{\tau(\ell)}} \rangle \\
&=  \sqrt{ \binom{k}{\vec{m}(\vec{i})} \binom{k}{\vec{m}(\vec{j})}} 
    \sum_{\sigma,\tau \in \S_k} 
        \prod_{\ell=1}^{k} [A]_{j_{\sigma(\ell)}, i_{\tau(\ell)}} \\
&=  \sqrt{ \binom{k}{\vec{m}(\vec{i})} \binom{k}{\vec{m}(\vec{j})}} 
    \sum_{\substack{ \vec{p}\in \Orb( \vec{i}) \\ \vec{q} \in \Orb(\vec{j})}} 
\binom{k}{\vec{m}(\vec{i})}^{-1} \binom{k}{\vec{m}(\vec{j})}^{-1}
        \prod_{\ell=1}^{k} [A]_{p_{\ell}, q_{\ell}} \\
&=  \frac{1}{ \sqrt{ \binom{k}{\vec{m}(\vec{i})} \binom{k}{\vec{m}(\vec{j})}} }
    \sum_{\substack{ \vec{p}\in \Orb( \vec{i}) \\ \vec{q} \in \Orb(\vec{j})}} 
        \prod_{\ell=1}^{k} [A]_{p_{\ell}, q_{\ell}} . \qedhere
\end{align*}
\end{proof}


We often identify $\V$ with $\R^n$ and $v_1 , v_2 ,\ldots, v_n$ with the standard basis.
Thus, we identify the linear operator $A$ on $\V$ with its matrix representation $[a_{ij}] \in \M_n$
with respect to $v_1 , v_2 ,\ldots, v_n $. Then $A^{\odot k}$ acts on the $N$-dimensional space $\V^{\odot k}$, in which 
$N = \binom{n+k-1}{k}$.  We identify $A^{\odot} \in \M_N$ with its matrix representation with respect to the orthonormal basis 
defined in Lemma \ref{Lemma:ONB}. 

\begin{example}\label{ex:Sym2Mat}
Let $n=k=2$ and $A = [a_{ij}] \in \M_2$.  
Then $A^{\odot 2} \in \M_3$ since $N=\binom{2+2-1}{2} =\binom{3}{2} = 3$.
Proposition \ref{Lemma:MatrixRep} and Table \ref{Table:n2k2}  produce
 \begin{equation}\label{eq:Ank22}
    A^{\odot 2}
    = 
    \begin{bmatrix}
    a_{11}^2 & a_{12}^2 &  \sqrt{2}a_{11}a_{12} \\[3pt]
    a_{21}^2 & a_{22}^2 &  \sqrt{2}a_{21}a_{22} \\[3pt]
    \sqrt{2} a_{11}a_{12} & \sqrt{2}a_{21}a_{22} &  a_{11}a_{22} +a_{12}a_{21}\\
    \end{bmatrix}.
\end{equation}
For example, 
{\small
\begin{equation*}
    [A^{\odot 2}]_{1,3}
    = \frac{1}{ \sqrt{ \frac{2!}{2!0!} \cdot \frac{2!}{1!1!} } } 
    \sum_{ \substack{ \vec{p} \in \{ (1,1)\} \\[2pt] \vec{q} \in \{ (1,2),(2,1) \} } }
    [A]_{p_1,q_1} [A]_{p_2,q_2} 
    = \frac{1}{ \sqrt{ 2 } }
    ( a_{11}a_{12} + a_{12}a_{11}) 
    = \sqrt{2} (a_{11} a_{12}).
\end{equation*}
}
 \begin{table}
        \begin{equation*}
       \begin{array}{|c|cccc|}
       \hline
        \text{index} &\vec{i} & \vec{m}(\vec{i}) & \binom{k}{\vec{m}(\vec{i})} & \Orb(\vec{i}) \\[3pt]
        \hline
        1 & (1,1) & (2,0) & \frac{2!}{2!0!} = 1 & \{ (1,1) \}\\[5pt]
        2 & (2,2) & (0,2) & \frac{2!}{0!2!} = 1 &\{ (2,2) \} \\[5pt]
        3 & (1,2) & (1,1) & \frac{2!}{1!1!} = 2 &\{ (1,2),(2,1) \}\\[3pt]
       \hline
        \end{array}
        \end{equation*}
        \caption{For $n=k=2$, there are $N=3$ nondecreasing elements of $[n]^k$.} 
        \label{Table:n2k2}
    \end{table}
 \end{example}

\begin{example}
Let $n=2$, $k=3$, and $A = [a_{ij}] \in \M_3$.   Then $A^{\odot 3} \in \M_4$ since $N = \binom{3+2-1}{3} = \binom{4}{3} = 4$.
Proposition \ref{Lemma:MatrixRep} and Table \ref{Table:n2k3}  produce
\begin{equation}\label{eq:Ank23}
A^{\odot 3} =
\begin{bmatrix}
a_{11}^{3} & a_{12}^{3} & \sqrt{3} a_{11}^{2} a_{12} & \sqrt{3} a_{11} a_{12}^{2} \\[3pt]
a_{21}^{3} & a_{22}^{3} & \sqrt{3} a_{21}^{2} a_{22} & \sqrt{3} a_{21} a_{22}^{2} \\[3pt]
\sqrt{3} a_{11}^{2} a_{21} & \sqrt{3} a_{12}^{2} a_{22} & 2 a_{11} a_{12} a_{21} + a_{11}^{2} a_{22} & a_{12}^{2} a_{21} + 2 a_{11} a_{12} a_{22} \\[3pt]
\sqrt{3} a_{11} a_{21}^{2} & \sqrt{3} a_{12} a_{22}^{2} & a_{12} a_{21}^{2} + 2 a_{11} a_{21} a_{22} & 2 a_{12} a_{21} a_{22} + a_{11} a_{22}^{2}
\end{bmatrix}.
\end{equation}
\noindent
For example,
\begin{align*}
[A^{\odot 3}]_{3,4}
    &= \frac{1}{ \sqrt{ \frac{3!}{2!1!} \cdot \frac{3!}{1!2!} } } 
    \sum_{ \substack{ \vec{p} \in \{(1,1,2),(1,2,1),(2,1,1)\}  \\[2pt] \vec{q} \in \{(1,2,2),(2,1,2),(2,2,1)\} }}
    [A]_{p_1,q_1} [A]_{p_2,q_2} [A]_{p_3,q_3} \\
    &= \frac{1}{3} ( a_{11}a_{12}a_{22} + a_{12} a_{11} a_{22} + a_{12}a_{12}a_{21} 
    +a_{11}a_{22}a_{12} + a_{12} a_{21} a_{12} \\
    &\quad\qquad  + a_{12} a_{22} a_{11} + a_{21} a_{12} a_{12} + a_{22} a_{11} a_{12} + a_{22} a_{12} a_{11}) \\
    &=a_{12}^{2} a_{21} + 2 a_{11} a_{12} a_{22}.
\end{align*}
 \begin{table}
 \captionsetup{width=.9\linewidth}
    \begin{equation*}
   \begin{array}{|c|cccc|}
   \hline
    \text{index} & \vec{i} & \vec{m}(\vec{i}) & \binom{k}{\vec{m}(\vec{i})} & \Orb(\vec{i}) \\
    \hline
    1 & (1,1,1) & (3,0) & \frac{3!}{3!0!} = 1 & \{(1,1,1) \}\\[5pt]
    2 & (2,2,2) & (0,3) & \frac{3!}{0!3!} = 1 & \{(2,2,2)\} \\[5pt]
    3 & (1,1,2) & (2,1) & \frac{3!}{2!1!} = 3 & \{(1,1,2),(1,2,1),(2,1,1)\}\\[5pt]
    4 & (1,2,2) & (1,2) & \frac{3!}{1!2!} = 3 & \{(1,2,2),(2,1,2),(2,2,1)\}\\[3pt]
   \hline
    \end{array}
    \end{equation*}
        \caption{For $n=2$ and $k=3$, there are $N=4$ nondecreasing
        elements of $[n]^k$.}
        \label{Table:n2k3}
    \end{table}
\end{example}

\begin{example}
Let $n=3$, $k=2$, and $A = [a_{ij}] \in \M_3$.  Then $A^{\odot 3} \in \M_6$ since
$N = \binom{3+2-1}{2} = \binom{4}{2} = 6$.
\noindent
Proposition \ref{Lemma:MatrixRep} and Table \ref{Table:n3k2} give
\begin{equation}\label{eq:Ank32}
A^{\odot{3}} = {\scriptsize
\begin{bmatrix}
a_{11}^{2} & a_{12}^{2} & a_{13}^{2} & \sqrt{2} a_{11} a_{12} & \sqrt{2} a_{11} a_{13} & \sqrt{2} a_{12} a_{13} \\[3pt]
a_{21}^{2} & a_{22}^{2} & a_{23}^{2} & \sqrt{2} a_{21} a_{22} & \sqrt{2} a_{21} a_{23} & \sqrt{2} a_{22} a_{23} \\[3pt]
a_{31}^{2} & a_{32}^{2} & a_{33}^{2} & \sqrt{2} a_{31} a_{32} & \sqrt{2} a_{31} a_{33} & \sqrt{2} a_{32} a_{33} \\[3pt]
\sqrt{2} a_{11} a_{21} & \sqrt{2} a_{12} a_{22} & \sqrt{2} a_{13} a_{23} & a_{12} a_{21} + a_{11} a_{22} & a_{13} a_{21} + a_{11} a_{23} & a_{13} a_{22} + a_{12} a_{23} \\[3pt]
\sqrt{2} a_{11} a_{31} & \sqrt{2} a_{12} a_{32} & \sqrt{2} a_{13} a_{33} & a_{12} a_{31} + a_{11} a_{32} & a_{13} a_{31} + a_{11} a_{33} & a_{13} a_{32} + a_{12} a_{33} \\[3pt]
\sqrt{2} a_{21} a_{31} & \sqrt{2} a_{22} a_{32} & \sqrt{2} a_{23} a_{33} & a_{22} a_{31} + a_{21} a_{32} & a_{23} a_{31} + a_{21} a_{33} & a_{23} a_{32} + a_{22} a_{33}
\end{bmatrix}}.
\end{equation}
\begin{table}
\captionsetup{width=.9\linewidth}
    \begin{equation*}
   \begin{array}{|c|cccc|}
   \hline
    \text{index} & \vec{i} & \vec{m}(\vec{i}) & \binom{k}{\vec{m}(\vec{i})} & \Orb(\vec{i}) \\
    \hline
    1 & (1,1) & (2,0,0) & \frac{2!}{2!0!0!} = 1 & \{ (1,1) \} \\[5pt]
    2 & (2,2) & (0,2,0) & \frac{2!}{0!2!0!} = 1 &\{(2,2)\} \\[5pt]
    3 & (3,3) & (0,0,2) & \frac{2!}{0!0!2!} = 1 &\{(3,3)\} \\[5pt]
    4 & (1,2) & (1,1,0) & \frac{2!}{1!1!0!} = 2 &\{(1,2),(2,1)\} \\[5pt]
    5 & (1,3) & (1,0,1) & \frac{2!}{1!0!1!} = 2 &\{(1,3),(3,1)\} \\[5pt]
    6 & (2,3) & (0,1,1) & \frac{2!}{0!1!1!} = 2 &\{(2,3),(3,2)\} \\[3pt]
   \hline
    \end{array}
    \end{equation*}
    \caption{For $n=3$ and $k=2$, there are $N=6$ nondecreasing
    elements of $[n]^k$.}
    \label{Table:n3k2}
\end{table}
\end{example}

\section{Symmetric tensor powers of graphs}\label{Section:Products}

Let $\G=(V,E,\omega)$ denote a weighted graph with vertex set $V = \{ v_1,v_2,\ldots,v_n\}$,
edge set $E$, and edge-weight function $\omega$;
nonexistent edges have weight $0$.
The vertices of $\G^{\odot k}$ are the $N = \binom{n+k-1}{k}$ nondecreasing 
elements $\vec{i}=(i_1,i_2,\ldots, i_k)$  of $[n]^k$, each of which we may identify with the
corresponding symmetric tensor $v_{i_1} \odot v_{i_2} \odot \cdots \odot v_{i_k}$ or its
normalization $u_{\vec{i}}$ in \eqref{eq:ONB}.
The edge weights are
\begin{equation}\label{eq:GraphWeight}
\omega^{\odot k} ( \vec{i}, \vec{j})
= \frac{1}{ \sqrt{ \binom{k}{\vec{m}(\vec{i})} \binom{k}{\vec{m}(\vec{j})}} }
\sum_{\substack{ \vec{p} \in \Orb( \vec{i}) \\ \vec{q} \in \Orb( \vec{j}) }} 
\prod_{\ell = 1}^k\omega (v_{p_\ell},v_{q_\ell}).
\end{equation}
The map $\G \mapsto \G^{\odot k}$ is well defined up to graph isomorphism;
see Appendix \ref{Section:WellDefined}.  If $A$ is the adjacency matrix of $\G$, then $A^{\odot k}$
is the adjacency matrix of $\G^{\odot k}$.

\begin{example}\label{Example:Complete}
Let $\G = \K_3$ be the complete graph on three vertices $v_1 , v_2 , v_3$. Its adjacency matrix
$A = \left[\begin{smallmatrix}  \0 & 1 & 1 \\ 1 & \0 & 1 \\ 1 & 1 & \0 \end{smallmatrix}\right]$
acts naturally on the inner-product space $\V$ with orthonormal basis $v_1 ,  v_2 ,  v_3$.
Then $\V^{\otimes 2}$ has orthonormal basis
\begin{equation*}
v_1 \otimes v_1, \,\,   v_1 \otimes v_2, \,\,   v_1 \otimes v_3, \,\,  
  v_2 \otimes v_1, \,\,   v_2 \otimes v_2, \,\,   v_2 \otimes v_3, \,\,  
   v_3 \otimes v_1, \,\,   v_3 \otimes v_2, \,\,   v_3 \otimes v_3.
\end{equation*}
The symmetric group $\S_2$ produces $N = \binom{3 +2-1}{2} = 6$ orbits in $[3]^2$, which yields
{\small
\begin{equation*}
 \{ v_1   \otimes v_1  \}, \,  \{ v_2   \otimes v_2  \}, \,  \{ v_3  \otimes v_3 \}, 
 \{ v_1   \otimes v_2 ,  v_2 \otimes v_1 \},  \{ v_1   \otimes v_3,  v_3 \otimes v_1 \} ,
  \{ v_2   \otimes v_3,  v_3 \otimes v_2  \}.
\end{equation*}
}%
Each orbit contains a nondecreasing representative; these are, respectively,
\begin{equation}
  v_1   \otimes v_1  , \quad   v_2   \otimes v_2  , \quad   v_3  \otimes v_3 , \quad 
  v_1   \otimes v_2  , \quad  v_1   \otimes v_3  ,\quad
   v_2   \otimes v_3  .
\end{equation}
The orthonormal basis provided by Lemma \ref{Lemma:ONB} is
\begin{align*}
  u_{(1,1)} & =   v_1 \odot v_1 ,
  &u_{(2,2)} & =  v_2 \odot  v_2  ,
  &u_{(3,3)} & =  v_3 \odot v_3,\\
  u_{(2,3)} & =  \sqrt{2} ( v_2 \odot v_3 ) ,
  &u_{(1,3)} & = \sqrt{2} ( v_1 \odot v_3)  , 
  &u_{(1,2)} & =   \sqrt{2} ( v_1 \odot  v_2)  ,
\end{align*}
and Lemma \ref{Lemma:MatrixRep} provides the adjacency matrix of $\K_3^{\odot 2}$:
\begin{equation*}
A^{\odot 2} = 
\left[
\begin{smallmatrix}
\0 & 1 & 1 & \0 & \0 & \sqrt{2} \\
1 & \0 & 1 & \0 & \sqrt{2} & \0 \\
1 & 1 & \0 & \sqrt{2} & \0 & \0 \\
\0 & \0 & \sqrt{2} & 1 & 1 & 1 \\
\0 & \sqrt{2} & \0 & 1 & 1 & 1 \\
\sqrt{2} & \0 & \0 & 1 & 1 & 1
\end{smallmatrix}
\right].
\end{equation*}
Note that $\K_3$ (Figure \ref{Figure:K3-1}) is a subgraph of $\K_3^{\odot 2}$ (Figure \ref{Figure:K3-2}).
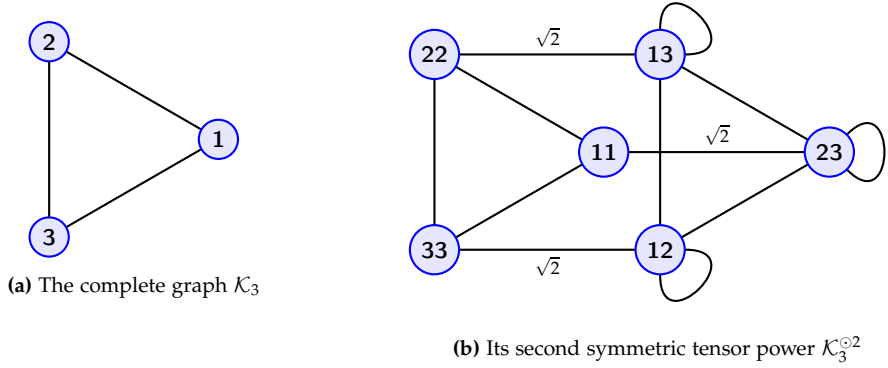
\begin{figure}
    \begin{subfigure}[c]{0.3\textwidth}
	                \centering
                        \begin{tikzpicture}[thick,scale=0.75, every node/.style={scale=0.75},auto, node distance=2cm,  thick,
                           node_style/.style={circle,draw=blue,fill=blue!10!,font=\sffamily\Large\bfseries},
                           edge_style/.style={draw=black,font=\small}]
                        \node[node_style] (v0) at (2,0) {1} ;
                        \node[node_style] (v1) at (-1,1.73205) {2} ;
                        \node[node_style] (v2) at (-1,-1.73205) {3} ;
                        \draw[edge_style]  (v0)--(v1);
                        \draw[edge_style]  (v0)--(v2);
                        \draw[edge_style]  (v1)--(v2);
                        \end{tikzpicture}
                        \caption{The complete graph $\K_3$} 
                        \label{Figure:K3-1}
	        \end{subfigure}
	        \hfill
	        \begin{subfigure}[c]{0.6\textwidth}
	                \centering
                        \begin{tikzpicture}[rotate=0,thick,scale=0.75, every node/.style={scale=0.75},auto, node distance=1.5cm,  thick,
                           node_style/.style={circle,draw=blue,fill=blue!10!,font=\sffamily\Large\bfseries},
                           edge_style/.style={draw=black,font=\small}]
                            \node[node_style] (v0) at (2,0) {11} ;
                            \node[node_style] (v1) at (-1,1.73205) {22} ;
                            \node[node_style] (v2) at (-1,-1.73205) {33} ;
                            \node[node_style] (v3) at (3,-1.73205) {12} ;
                            \node[node_style] (v4) at (3,1.73205) {13} ;
                            \node[node_style] (v5) at (6,0) {23} ;
                            \draw[edge_style]  (v0)--(v1);
                            \draw[edge_style]  (v0)--(v2);
                            \draw[edge_style]  (v0) edge node{$ \sqrt{2} $} (v5);
                            \draw[edge_style]  (v1)--(v2);
                            \draw[edge_style]  (v1) edge node{$ \sqrt{2} $} (v4);
                            \draw[edge_style]  (v3) edge node{$ \sqrt{2} $} (v2);
                            \draw[edge_style]  (v3)--(v4);
                            \draw[edge_style]  (v3)--(v5);
                            \draw[edge_style]  (v4)--(v5);
                            \draw (v3) to [out=0,in=-90,looseness=5] (v3) ;
                            \draw (v4) to [out=0,in=90,looseness=5] (v4) ;
                            \draw (v5) to [out=45,in=-45,looseness=5] (v5) ;
                        \end{tikzpicture}
                        \vspace{-10pt}
                        \caption{Its second symmetric tensor power $\K_3^{\odot 2}$}        
                        \label{Figure:K3-2}      
	        \end{subfigure}
    \captionsetup{width=.95\linewidth}	        
    \caption{The complete graph $\K_3$ and its second symmetric tensor power.  Symmetric tensor powers of unweighted loopless
    graphs can be weighted and may have loops. }
    \label{Figure:Complete}
\end{figure}
\end{example}

\begin{example}\label{Example:Path}
\begin{figure}
\centering		
		\begin{subfigure}[t]{0.125\textwidth}
	                \centering
                       \begin{tikzpicture}[thick,scale=0.75, yscale=0.7, every node/.style={scale=0.75},auto, node distance=1,  thick,
                           node_style/.style={circle,draw=blue,fill=blue!10!,font=\sffamily\Large\bfseries},
                           edge_style/.style={draw=black,font=\small}]
                        
                            \node[node_style] (v1) at (-2,2) {3};
                            \node[node_style] (v2) at (-2,-1) {2};
                            \node[node_style] (v3) at (-2,-4) {1};
                            \draw[edge_style]  (v1)--(v2);
                            \draw[edge_style]  (v2)--(v3);
                            \end{tikzpicture}
                        \caption{$\P_3$}   
                        \label{Figure:Path1}           
       	        \end{subfigure}
	        \quad
       	        \begin{subfigure}[t]{0.4\textwidth}
	                \centering
                        \begin{tikzpicture}[thick,scale=0.75, every node/.style={scale=0.75},auto, node distance=2.5,  thick,
                           node_style/.style={circle,draw=blue,fill=blue!10!,font=\sffamily\Large\bfseries},
                           edge_style/.style={draw=black,font=\small}]
                        
                            \node[node_style] (v1) at (-2,-2) {11};
                            \node[node_style, circle,draw=green,fill=green!10] (v2) at (0,-2) {12};
                            \node[node_style, circle,draw=orange,fill=orange!10] (v3) at (2,-2) {13};
                            \node[node_style, circle,draw=green,fill=green!10] (v4) at (-2,0) {21};
                            \node[node_style] (v5) at (0,0) {22};
                            \node[node_style, circle,draw=red,fill=red!10!] (v6) at (2,0) {23};
                            \node[node_style, circle,draw=orange,fill=orange!10!] (v7) at (-2,2) {31};
                            \node[node_style, circle,draw=red,fill=red!10!] (v8) at (0,2) {32};
                            \node[node_style] (v9) at (2,2) {33};
                            \draw[edge_style]  (v1)--(v5);
                            \draw[edge_style]  (v5)--(v9);  
                            \draw[edge_style]  (v5)--(v3);
                            \draw[edge_style]  (v5)--(v7);
                            \draw[edge_style]  (v2)--(v6);
                            \draw[edge_style]  (v2)--(v4);
                            \draw[edge_style]  (v8)--(v4);
                            \draw[edge_style]  (v8)--(v6);                            
                            \end{tikzpicture}
                        \caption{$\P_3^{\otimes 2}$}              
                        \label{Figure:Path2}           
	        \end{subfigure}
	        \quad
	        \begin{subfigure}[t]{0.4\textwidth}
	                \centering
                       \begin{tikzpicture}[thick,scale=0.75, every node/.style={scale=0.75},auto, thick,
                           node_style/.style={circle,draw=blue,fill=blue!10!,font=\sffamily\Large\bfseries},
                           edge_style/.style={draw=black,font=\small}]
                        
                            \node[node_style] (v11) at (-2,1.9) {11};
                            \node[node_style] (v22) at (-2,0) {22};
                      \node[node_style] (v33) at (-2,-1.9) {33};      
                      \node[node_style,draw=green,fill=green!10] (v12) at (2.25,-1.25) {12};
                      \node[node_style,draw=orange,fill=orange!10] (v13) at (0.5,0) {13};
                      \node[node_style,draw=red,fill=red!10] (v23) at (2.25,1.25) {23};      
                      \draw[edge_style]  (v11)--(v22);
                      \draw[edge_style]  (v22)--(v33);                      
                      \draw[edge_style]  (v22) edge node{$\sqrt{2}$} (v13);
                      \draw[edge_style]  (v12)--(v23);                           
                            \draw (v12) to [out=180+45,in=180-45,looseness=5] (v12) ;
                            \draw (v23) to [out=180+45,in=180-45,looseness=5] (v23) ;
                            \end{tikzpicture}
                        \caption{$\P_3^{\odot 2}$}              
                        \label{Figure:Path3}           
	        \end{subfigure}
                 \caption{The symmetric tensor power $\P_3^{\odot 2}$ of the path graph $\P_3$
                 is obtained by identifying vertices in the tensor (Kronecker) power $\P_3^{\otimes 2}$
                 and adding the appropriate weights. Edge weights equal to $1$ are suppressed.}              
    \label{Figure:Compress}
\end{figure}
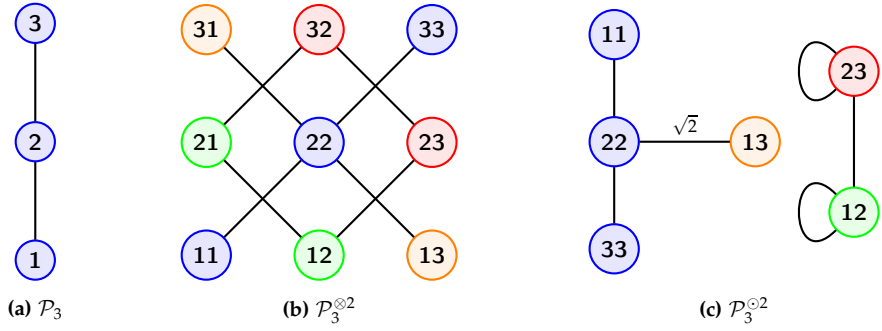
Let $\P_3$ denote the path graph on three vertices (Figure \ref{Figure:Path1}).
Its second symmetric tensor power $\P_3^{\odot 2}$ (Figure \ref{Figure:Path3}) 
is obtained from the tensor (Kronecker) power $\P^{\otimes 2}$ (Figure \ref{Figure:Path2})
by identifying ordered pairs in the same $\S_2$ orbit, selecting nondecreasing representatives, 
and adding weights according to \eqref{eq:GraphWeight}.
\end{example}

Examples \ref{Example:Complete} and \ref{Example:Path} suggest that an (unweighted) graph
is a subgraph of its symmetric tensor powers.  The next theorem establishes this fact.

\begin{theorem}\label{Theorem:Subgraph}
If $\G$ is an unweighted graph and $k \in \N$, then $\G$ is a subgraph of $\G^{\odot k}.$
\end{theorem}

\begin{proof}
Let $\G$ be an unweighted graph on $n$ vertices, with adjacency matrix $A$.
Let $\vec{i} = (i,i,\ldots,i) \in [n]^k$ and $\vec{j} = (j,j,\ldots,j) \in [n]^k$.
Then $\vec{m}(\vec{i})$ and $\vec{m}(\vec{j})$ are $k$ times the standard basis vectors in $\R^n$
that have their $1$s in the $i$th and $j$th positions, respectively.  
Since $\Orb( \vec{i} ) = \{ \vec{i}\}$ and $\Orb( \vec{j} ) = \{ \vec{j} \}$,
Proposition \ref{Lemma:MatrixRep} ensures that
\begin{equation*}
[A^{\odot k}]_{\vec{i}, \vec{j}}
= \frac{1}{ \sqrt{ \binom{k}{\vec{m}(\vec{i})} \binom{k}{\vec{m}(\vec{j})}} }
    \sum_{\substack{ \vec{p} \in \Orb( \vec{i}) \\ \vec{q} \in \Orb( \vec{j}) }} 
        \prod_{\ell=1}^{k} [A]_{p_{\ell}, q_{\ell}} 
=     \prod_{\ell=1}^{k} [A]_{i,j} = [A]_{i,j}
\end{equation*}
because the matrix entries of $A$ belong to $\{0,1\}$.
\end{proof}

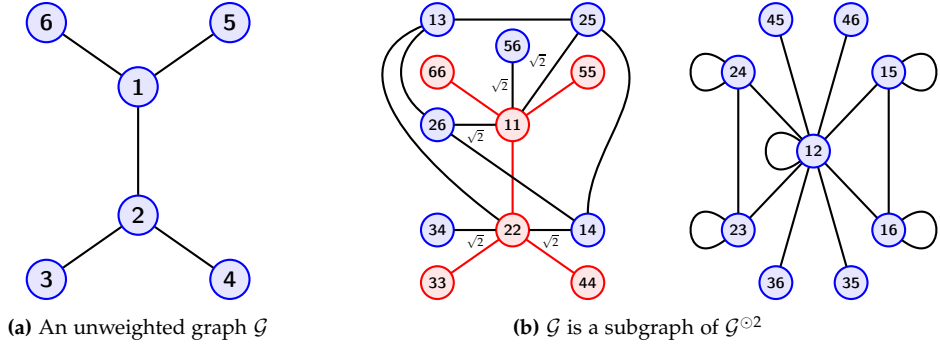
\begin{figure}    
\centering
\begin{subfigure}[t]{0.275\textwidth}\centering
    \begin{tikzpicture}[thick,scale=1.22, yscale=0.7, every node/.style={scale=0.75},auto, node distance=2cm,  thick,                           node_style/.style={circle,draw=blue,fill=blue!10!,font=\sffamily\Large\bfseries},                          edge_style/.style={draw=black,font=\small}]
    \node[node_style] (v0) at (0,1) {1} ;
    \node[node_style] (v1) at (0,-1) {2} ;
    \node[node_style] (v2) at (-1,-2) {3} ;
    \node[node_style] (v3) at (1,-2) {4} ;
    \node[node_style] (v4) at (1,2) {5} ;
    \node[node_style] (v5) at (-1,2) {6} ;
    \draw[edge_style]  (v0)--(v1);
    \draw[edge_style]  (v0)--(v4);
    \draw[edge_style]  (v0)--(v5);
    \draw[edge_style]  (v1)--(v2);
    \draw[edge_style]  (v1)--(v3);
\end{tikzpicture} 
\caption{An unweighted graph $\G$}                         
\label{Figure:Barbell1}	        
\end{subfigure}
\hfill
\begin{subfigure}[t]{0.675\textwidth}\centering
    \begin{tikzpicture}[thick,scale=1.0, yscale=0.7,every node/.style={scale=0.5},auto, node distance=2cm,  thick,                           node_style/.style={circle,draw=blue,fill=blue!10!,font=\sffamily\Large\bfseries},                          edge_style/.style={draw=black,font=\small}]
    \node[node_style, fill=red!10, draw=red] (v0) at (0,1) {11} ;
    \node[node_style, fill=red!10, draw=red] (v1) at (0,-1) {22} ;
    \node[node_style, fill=red!10, draw=red] (v2) at (-1,-2) {33} ;
    \node[node_style, fill=red!10, draw=red] (v3) at (1,-2) {44} ;
    \node[node_style, fill=red!10, draw=red] (v4) at (1,2) {55} ;
    \node[node_style, fill=red!10, draw=red] (v5) at (-1,2) {66} ;
    \node[node_style] (v6) at (4,0.5) {12} ;
    \node[node_style] (v7) at (-1,3) {13} ;
    \node[node_style] (v8) at (1,-1) {14} ;
    \node[node_style] (v9) at (5,2) {15} ;
    \node[node_style] (v10) at (5,-1) {16} ;
    \node[node_style] (v11) at (3,-1) {23} ;
    \node[node_style] (v12) at (3,2) {24} ;
    \node[node_style] (v13) at (1,3) {25} ;
    \node[node_style] (v14) at (-1,1) {26} ;
    \node[node_style] (v15) at (-1,-1) {34} ;
    \node[node_style] (v16) at (4.5,-2) {35} ;
    \node[node_style] (v17) at (3.5,-2) {36} ;
    \node[node_style] (v18) at (3.5,3) {45} ;
    \node[node_style] (v19) at (4.5,3) {46} ;
    \node[node_style] (v20) at (0,2.5) {56} ;
    \draw[edge_style, draw=red]  (v0)--(v1);
    \draw[edge_style, draw=red]  (v0)--(v4);
    \draw[edge_style]  (v0) edge node{$ \sqrt{2} $} (v20);
    \draw[edge_style, draw=red]  (v0)--(v5);
    \draw[edge_style]  (v0) edge node{$ \sqrt{2} $} (v13);
    \draw[edge_style]  (v0) edge node{$ \sqrt{2} $} (v14);
    \draw[edge_style, draw=red]  (v1)--(v2);
    \draw[edge_style, draw=red]  (v1)--(v3);
    \draw (v7) to [out=210, in=135, looseness=1] (v1);
    \draw[edge_style]  (v8) edge node{$ \sqrt{2} $} (v1);
    \draw[edge_style]  (v1) edge node{$ \sqrt{2} $} (v15);
    \draw[edge_style]  (v6)--(v16);
    \draw[edge_style]  (v6)--(v17);
    \draw[edge_style]  (v6)--(v18);
    \draw[edge_style]  (v6)--(v19);
    \draw[edge_style]  (v6)--(v9);
    \draw[edge_style]  (v6)--(v10);
    \draw[edge_style]  (v6)--(v11);
    \draw[edge_style]  (v6)--(v12);
    \draw[edge_style]  (v7)--(v13);
    \draw (v7) to [out=-135, in=135, looseness=1] (v14);
    \draw (v8) to [out=90, in=-45, looseness=1] (v13);
    \draw[edge_style]  (v8)--(v14);
    \draw[edge_style]  (v9)--(v10);
    \draw[edge_style]  (v11)--(v12);
    \draw (v6) to [out=135,in=-135,looseness=5] (v6) ;
    \draw (v9) to [out=45,in=-45,looseness=5] (v9) ;
    \draw (v10) to [out=45,in=-45,looseness=5] (v10) ;
    \draw (v11) to [out=135,in=-135,looseness=5] (v11) ;
    \draw (v12) to [out=135,in=-135,looseness=5] (v12) ;
\end{tikzpicture} 
\caption{$\G$ is a subgraph of $\G^{\odot 2}$}                         
\label{Figure:Barbell2}	        
\end{subfigure}
\caption{Illustration of Theorem \ref{Theorem:Subgraph}.}    
\label{Figure:Barbell}
\end{figure}

The previous theorem is illustrated in Figures \ref{Figure:Barbell} and \ref{Figure:Scepter}.

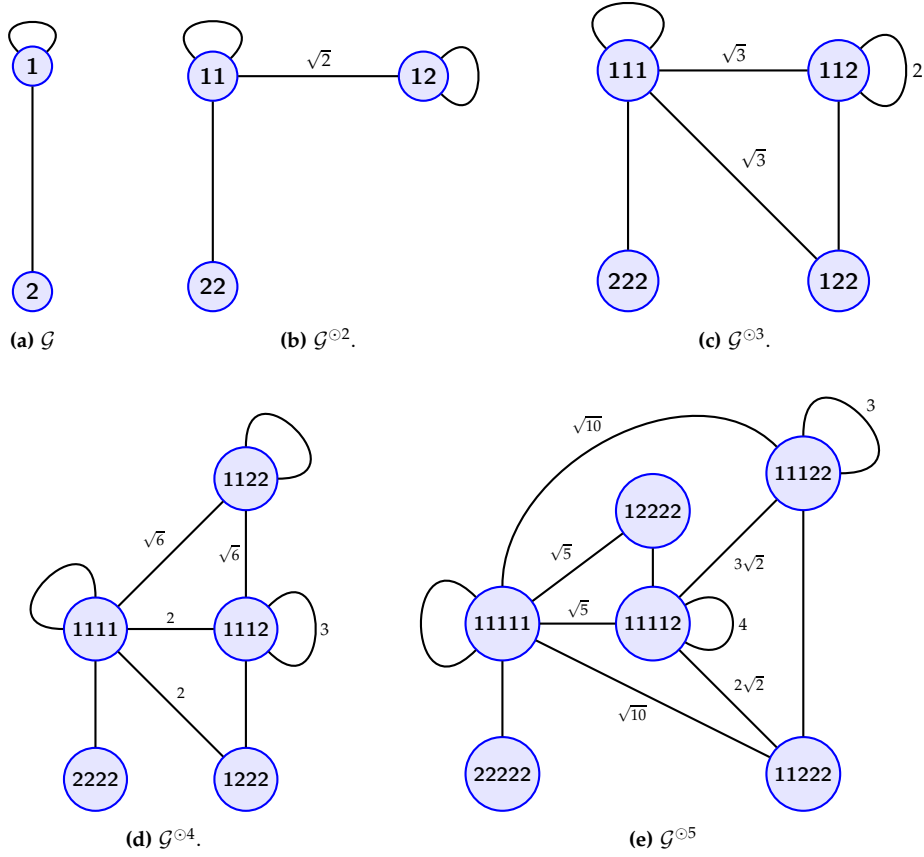
\begin{figure}[h]
\centering
		\begin{subfigure}[t]{0.125\textwidth}
	                \centering
                        \begin{tikzpicture}[thick,scale=0.75, every node/.style={scale=0.75},auto, thick,
                           node_style/.style={circle,draw=blue,fill=blue!10!,font=\sffamily\Large\bfseries},
                           edge_style/.style={draw=black,font=\small}]                        
                            \node[node_style] (v1) at (0,2) {1};
                            \node[node_style] (v2) at (0,-2) {2};
                            \draw[edge_style]  (v1)--(v2);
                            \draw[black] (v1) to [out=45,in=135,looseness=5] (v1) ;
                            \end{tikzpicture}
                        \caption{$\G$}
                        \label{Figure:Scepter1}              
	        \end{subfigure}                	
\hfill     
		\begin{subfigure}[t]{0.4\textwidth}
	                \centering
                       \begin{tikzpicture}[thick,scale=0.7, every node/.style={scale=0.75},auto, thick,
                           node_style/.style={circle,draw=blue,fill=blue!10!,font=\sffamily\Large\bfseries},
                           edge_style/.style={draw=black,font=\small}]
                        
                            \node[node_style] (v1) at (-2,2) {11};
                            \node[node_style] (v2) at (-2,-2) {22};
                            \node[node_style] (v3) at (2,2) {12};
                            \draw[edge_style]  (v1)--(v2);
                            \draw[edge_style]  (v1) edge node{$\sqrt{2}$} (v3);

                            \draw[black] (v1) to [out=45,in=135,looseness=5] (v1) ;
                            \draw[black] (v3) to [out=45,in=-45,looseness=5] (v3) ;
                            \end{tikzpicture}
                        \caption{$\G^{\odot 2}$.}              
                        \label{Figure:Scepter2}              
       	        \end{subfigure}
\hfill	        
\begin{subfigure}[t]{0.4\textwidth}	                
    \centering                        
    \begin{tikzpicture}[thick,scale=0.7, every node/.style={scale=0.75},auto, thick,                           node_style/.style={circle,draw=blue,fill=blue!10!,font=\sffamily\Large\bfseries},                          edge_style/.style={draw=black,font=\small}]
    \node[node_style] (v0) at (-2,2) {111} ;
    \node[node_style] (v1) at (-2,-2) {222} ;
    \node[node_style] (v2) at (2,2) {112} ;
    \node[node_style] (v3) at (2,-2) {122} ;
    \draw[edge_style]  (v0)--(v1);
    \draw[edge_style]  (v0) edge node{$ \sqrt{3} $} (v2);
    \draw[edge_style]  (v0) edge node{$ \sqrt{3} $} (v3);
    \draw[edge_style]  (v2)--(v3);
    \draw[black] (v0) to [out=135,in=45,looseness=5] (v0) ;
    \draw[black] (v2) to [out=45,in=-45,looseness=5]  (v2) ;
    \node at (3.5,2){$2$};
    \end{tikzpicture} 
    \caption{$\G^{\odot 3}$.}              
    \label{Figure:Scepter3}              
\end{subfigure}
\\[-10pt]
\begin{subfigure}[t]{0.4\textwidth}	                
    \centering                        
    \begin{tikzpicture}[scale=0.5, every node/.style={scale=0.65},auto, node distance=2,  thick,                           node_style/.style={circle,draw=blue,fill=blue!10!,font=\sffamily\Large\bfseries},                          edge_style/.style={draw=black,font=\small}]
    \node[node_style] (v0) at (-2,2) {1111} ;
    \node[node_style] (v1) at (-2,-2) {2222} ;
    \node[node_style] (v2) at (2,2) {1112} ;
    \node[node_style] (v3) at (2,6) {1122} ;
    \node[node_style] (v4) at (2,-2) {1222} ;
    \draw[edge_style]  (v0)--(v1);
    \draw[edge_style]  (v0) edge node{$2$} (v2);
    \draw[edge_style]  (v0) edge node{$\sqrt{6}$} (v3);
    \draw[edge_style]  (v0) edge node{$2$} (v4);
    \draw[edge_style]  (v2) edge node{$\sqrt{6}$} (v3);
    \draw[edge_style]  (v2)--(v4);
    \draw[black] (v0) to [out=90,in=180,looseness=5] (v0) ;
    \draw[black] (v2) to [out=45,in=-45,looseness=5] (v2) ;
    \node at (4.1,2){$3$};
    \draw[black] (v3) to [out=0,in=90,looseness=5] (v3) ;
    \end{tikzpicture} 
    \caption{$\G^{\odot 4}$.}              
    \label{Figure:Scepter4}              
\end{subfigure}	        
\hfill
\begin{subfigure}[t]{0.55\textwidth}
    \centering
    \begin{tikzpicture}[thick,scale=0.5, every node/.style={scale=0.65},auto, node distance=2,  thick,                           node_style/.style={circle,draw=blue,fill=blue!10!,font=\sffamily\Large\bfseries},                          edge_style/.style={draw=black,font=\small}]
    \node[node_style] (v0) at (-2,2) {11111} ;
    \node[node_style] (v1) at (-2,-2) {22222} ;
    \node at (4.4,2){$4$};
    \node[node_style] (v2) at (2,2) {11112} ;
    \node[node_style] (v3) at (6,6) {11122} ;
    \node at (7.8,7.8){$3$};
    \node[node_style] (v4) at (6,-2) {11222} ;
    \node[node_style] (v5) at (2,5) {12222} ;
    \draw[edge_style]  (v0)--(v1);
    \draw[edge_style]  (v0) edge node{$ \sqrt{5} $} (v2);
    \draw[edge_style, in=135, out=90, looseness=1]  (v0) edge node{$ \sqrt{10} $} (v3);
    \draw[edge_style]  (v4) edge node{$ \sqrt{10} $} (v0);
    \draw[edge_style]  (v0) edge node{$ \sqrt{5} $} (v5);
    \draw[edge_style]  (v3) edge node{$3\sqrt{2}$} (v2);
    \draw[edge_style]  (v2) edge node{$2\sqrt{2}$} (v4);
    \draw[edge_style]  (v2)--(v5);
    \draw[edge_style]  (v3)--(v4);
    \draw[black] (v0) to [out=135,in=-135,looseness=5] (v0) ;
    \draw[black] (v2) to [out=30,in=-30,looseness=5] (v2) ;
    \draw[black] (v3) to [out=0,in=90,looseness=5] (v3) ;
    \end{tikzpicture} 
    \caption{$\G^{\odot 5}$}   
    \label{Figure:Scepter5}               
\end{subfigure}	        
    \captionsetup{width=.95\linewidth}	        
\caption{Symmetric tensor powers of the ``scepter'' graph $\G$;
edge weights with value $1$ are suppressed.  Theorem \ref{Theorem:Subgraph} ensures that $\G$ appears
as a subgraph in $\G^{\odot k}$ for $k \in \N$.}                
\label{Figure:Scepter}
\end{figure}

\begin{proposition}
    (a) Let $\G$ be a graph that contains a loop, then as unweighted graphs 
    $\G^{\odot k_1}$ is a subgraph of $\G^{\odot k_2}$ for all $k_1\leq k_2$.
    (b) Let $\G$ be a nonempty graph, then as unweighted graphs $\G^{\odot k_1}$ is a subgraph of $\G^{\odot k_1+2\cdot \ell}$ for all $\ell \geq 0$.
\end{proposition}

\begin{proof}
    (a) We may assume the loop is at vertex $v_1$. 
    Notice that if $v_i$ is adjacent to $v_j$ in $\G^{\odot k_1}$, then $\underbrace{v_1 \odot v_1 \odot \cdots \odot v_1}_{k_2-k_1} \odot v_i$ 
    and $\underbrace{v_1 \odot v_1 \odot \cdots \odot v_1}_{k_2-k_1} \odot v_j$ are adjacent.
    
    \noindent
    (b) We may assume $v_1$ and $v_2$ are adjacent in $\G$. 
    If $v_i$ is adjacent to $v_j$ in $\G^{\odot k_1}$, then $\underbrace{v_1 \odot v_2\odot \cdots \odot v_1\odot v_2}_{\frac{k_2-k_1}{2}}\odot v_i$ and $\underbrace{v_1\odot v_2\odot \cdots \odot v_1\odot v_2}_{\frac{k_2-k_1}{2}}\odot v_j$ are adjacent. 
\end{proof}

\section{Compatibility with graph spectra}\label{Section:Spectrum}

The \emph{spectrum} of a graph is the spectrum (multiset of eigenvalues) of its adjacency matrix.  This is
well defined since isomorphic graphs correspond to permutation-similar adjacency matrices,
and similar matrices have the same the spectrum.
Symmetric tensor powers respect graph spectra.  Although the next result is familiar
to representation theorists and mathematical physicists,
we prove it for the sake of graph theorists to whom it may be novel.

\begin{theorem}\label{Theorem:Eigenvalues}
Let $\G$ be a graph on $n$ vertices whose adjacency matrix $A$ has spectrum $\lambda_1,\lambda_2,\ldots,\lambda_n$.
Then the spectrum of $\G^{\odot k}$ consists of the $\binom{n+k-1}{k}$ products
$\lambda_{i_1}\lambda_{i_2}\cdots \lambda_{i_k}$ with $1 \leq i_1 \leq i_2 \leq \cdots \leq i_k \leq n$;
these are all of the eigenvalues of $A^{\odot k}$.
\end{theorem}

\begin{proof}
Let $\G$ have vertices $v_1 , v_2 ,\ldots,v_n $.
For $1 \leq i_1 \leq i_2 \leq \cdots \leq i_k \leq n$, 
\begin{align*}
A^{\odot k}(v_{i_1} \odot v_{i_2} \odot \cdots \odot v_{i_k} )
&= \frac{1}{k!} \sum_{\sigma \in \S_k} (Av_{i_{\sigma(1)}}) \otimes (Av_{i_{\sigma(2)}}) \otimes \cdots \otimes (Av_{i_{\sigma(k)}}) \\
&= \frac{1}{k!} \sum_{\sigma \in \S_k} (\lambda_{i_1} v_{i_{\sigma(1)}}) \otimes (\lambda_{i_2} v_{i_{\sigma(2)}} )\otimes \cdots \otimes (\lambda_{i_k} v_{i_{\sigma(k)}} )\\
&= \lambda_{i_1}\lambda_{i_2}\cdots \lambda_{i_k} \bigg(\frac{1}{k!} \sum_{\sigma \in \S_k} v_{i_{\sigma(1)}} \otimes v_{i_{\sigma(2)}} \otimes \cdots \otimes v_{i_{\sigma(k)}} \bigg)\\
&= \lambda_{i_1}\lambda_{i_2}\cdots \lambda_{i_k}(v_{i_1} \odot v_{i_2} \odot \cdots \odot v_{i_k}).
\end{align*}
Thus, $v_{i_1} \odot v_{i_2} \odot \cdots \odot v_{i_k}$ is an eigenvector of $A^{\odot k}$ with eigenvalue
$\lambda_{i_1}\lambda_{i_2}\cdots \lambda_{i_k}$.  
There are no others since the $v_{i_1} \odot v_{i_2} \odot \cdots \odot v_{i_k}$ form 
 a basis for $\V^{\odot k}$.
\end{proof}

\begin{example}\label{Example:Fibonacci}
Figure \ref{Figure:Scepter} illustrates the graphs $\G^{\odot k}$ for $k=2,3,4,5$, in which 
$\G$ is the scepter graph (Figure \ref{Figure:Scepter1}).  The corresponding adjacency matrices are
\begin{equation*}
A = \displaystyle \minimatrix{1}{1}{1}{\0},
\quad
A^{\odot 2} =
\begin{bmatrix}
 1 & 1 & \sqrt{2} \\
 1 & \0 & \0 \\
 \sqrt{2} & \0 & 1 \\
\end{bmatrix},
\quad
A^{\odot 3} =\left[
\begin{smallmatrix}
 1 & 1 & \sqrt{3} & \sqrt{3} \\
 1 & \0 & \0  & \0\\
 \sqrt{3} & \0 & 2 & 1 \\
 \sqrt{3} & \0 & 1 & \0
\end{smallmatrix}\right],
\end{equation*}
\begin{equation*}
A^{\odot 4}=
\left[
\begin{smallmatrix}
 1 & 1 & 2 & \sqrt{6} & 2 \\
 1 & \0 & \0 & \0 & \0 \\
 2 & \0 & 3 & \sqrt{6} & 1 \\
 \sqrt{6} & 0 & \sqrt{6} & 1 & \0 \\
 2 & \0 & 1 & \0 & \0 \\
\end{smallmatrix}
\right],
\quad
A^{\odot 5}=
\left[
\begin{smallmatrix}
 1 & 1 & \sqrt{5} & \sqrt{10} & \sqrt{10} & \sqrt{5} \\
 1 & \0 & \0 & \0 & \0 & \0 \\
 \sqrt{5} & \0 & 4 & 3 \sqrt{2} & 2 \sqrt{2} & 1 \\
 \sqrt{10} & \0 & 3 \sqrt{2} & 3 & 1 & \0 \\
 \sqrt{10} & \0 & 2 \sqrt{2} & 1 & \0 & \0 \\
 \sqrt{5} & \0 & 1 & \0 & \0 & \0 \\
\end{smallmatrix}
\right].
\end{equation*}
The respective graph spectra are
\begin{align*}
\G: &\quad \tfrac{1}{2}(1-\sqrt{5}),\quad \tfrac{1}{2}(1+\sqrt{5}), \\
\G^{\odot 2}: &\quad -1,\quad\tfrac{1}{2} (3-\sqrt{5}),\quad\tfrac{1}{2} (\sqrt{5}+3), \\
\G^{\odot 3}: &\quad \tfrac{1}{2} (-1-\sqrt{5}),\quad  2-\sqrt{5},  \quad \tfrac{1}{2} (-1+\sqrt{5}), \quad \sqrt{5}+2 , \\
\G^{\odot 4}: &\quad 1,\quad\tfrac{1}{2} (7-3 \sqrt{5}),\quad\tfrac{1}{2} (-\sqrt{5}-3),\quad
\tfrac{1}{2} (\sqrt{5}-3),\quad\tfrac{1}{2} (3 \sqrt{5}+7), \\
\G^{\odot 5}: &\quad\tfrac{1}{2} (11-5 \sqrt{5}),\,\,-\sqrt{5}-2,\,\,\tfrac{1}{2} (1-\sqrt{5}),\,\,\sqrt{5}-2,\,\,\tfrac{1}{2} (\sqrt{5}+1),\,\,\tfrac{1}{2} (5 \sqrt{5}+11).
\end{align*}
These agree with the results  of Theorem \ref{Theorem:Eigenvalues}.
The eigenvalues $\lambda_1 = \frac{1}{2}(1-\sqrt{5})$ and $\lambda_2 = \frac{1}{2}(1+\sqrt{5})$ of $A$
give rise to the eigenvalues of $A^{\odot 2}$:
\begin{equation*}
\lambda_1^2 = \tfrac{1}{2} (3-\sqrt{5}),\qquad
\lambda_1 \lambda_2 = -1, \qquad \text{and} \qquad
\lambda_2^2 = \tfrac{1}{2} (3+\sqrt{5}).
\end{equation*}
A similar calculation shows that  the eigenvalues of $A^{\odot 3}$ are 
\begin{equation*}
\lambda_1^3 =  2-\sqrt{5}, \!\quad
\lambda_1^2 \lambda_2 =  \tfrac{1}{2} (-1+\sqrt{5}), \quad\!
\lambda_1 \lambda_2^2 = \tfrac{1}{2} (-1-\sqrt{5}), \quad\! \text{and} \quad\!
\lambda_2^3 = \sqrt{5}+2.
\end{equation*}
\end{example}

Theorem \ref{Theorem:Eigenvalues} provides convenient formulas for the trace and determinant of a graph, that is, the trace and determinant of its adjacency matrix.

\begin{corollary}
Let $\G$ be a graph with spectrum $\lambda_1,\lambda_2,\ldots,\lambda_n$ and let $k \in \N$.
Then $\det( \G^{\odot k}) = (\det \G)^{\binom{n+k-1}{n}}$ and 
$\tr (\G^{\odot k}) = \sum_{i_1\leq i_2 \leq \cdots \leq i_k} \lambda_{i_1} \lambda_{i_2} \cdots \lambda_{i_k}$, 
the complete homogeneous symmetric polynomial of degree $k$ in the eigenvalues of $\G$.
\end{corollary}

\begin{proof}
The trace of a graph $\G$ is the sum of the eigenvalues of its adjacency matrix $A$, so the desired trace formula follows from Theorem \ref{Theorem:Eigenvalues}.  The determinant is the product of the eigenvalues, so
\begin{align*}
  \det  (A^{\odot k} )
  = \prod _{\substack{\vec{i} = (i_1,i_2,\ldots,i_k)\\ 1\leq i_1 \leq i_2 \leq \cdots \leq i_k \leq n}}\,\,
  \prod _{j=1}^n\lambda _j^{\vec{m}(\vec{i})_j}=\prod _{i=1}^{n}\lambda _i^{\beta _i},  
\end{align*}
in which $\beta _i$ is the number of times $\lambda _i$ appears in the whole product. This is equivalent to adding the contribution of each individual product, so 
\begin{equation*}
\beta _i=\sum _{\ell =0}^k\ell \cdot \binom{k-\ell +n-2}{k-\ell}=\binom{k+n-1}{n},
\end{equation*}
in which the final equality arises as follows: let $(x_1,x_2,\cdots ,x_n)$ be a composition of $k$ into $n$ nonnegative parts. Fix $i\in [n]$ such that $x_i=\ell$ and take out $0\leq r<\ell$ from  the part $x_i$ to a new part $x_{n+1}$. 
There are $\ell$ choices for $r$ to take to the new part $x_{n+1}$ and it cannot be $0$ so $x _{n+1}>0$, so we have taken 
one out, which adds up to $k-1$, and so these are compositions of $k-1$ into $n+1$ nonnegative parts which are counted 
by $\binom{(k-1)+n+1-1}{n+1-1}=\binom{n+k-1}{n}$.
\end{proof}

\section{Combinatorial properties}\label{Section:Combinatorial}
The number $|E_{\G^{\otimes k}}|$ of edges of a tensor (Kronecker) power $\G^{\otimes k}$ of an undirected simple graph $\G$ is related to the original number of edges $| E_{\G} |$ by $ | E_{\G^{\otimes k}} |=2^{k-1} | E_{\G} |^k$ \cite{hammack-2011a}.  This yields the upper bound
in the next result. 
The lower bound is obtained by noticing that an edge in $\G^{\otimes k}$ can be seen as an edge in $\G^{\otimes k}$ when one permute the indices of one of the vertices, so $|E_{\G^{\otimes k}}|\leq |E_{\G^{\odot k}}|\cdot k!$.

\begin{proposition}
For a simple undirected graph, $\G$ we have
\begin{equation*}
\frac{2^{k-1}|E_{\G}|^k}{k!} \leq |E_{\G^{\odot k}}|\leq 2^{k-1}|E_{\G}|^k.
\end{equation*}
\end{proposition}

The number of edges adjacent to a vertex $v$ of a graph $\G$ is the \emph{degree} of $v$ in $\G$, denoted  $\deg_{\G} v$. 
This is the size of the set $N_{\G}(v)=\{u\in V_{\G}:\{u,v\}\in E_{\G}\}$ of neighboring vertices.  The next 
proposition characterizes the neighbor set of a vertex and gives a bound on its cardinality. 
 The cartesian product is used to generate sequences from the neighbours of elements in the vertex.
 
\begin{proposition}\label{Prop:Degree}
Let $\G$ an undirected graph with vertex set $V_{\G} =\{v_1,v_2,\ldots ,v_n\}$ and let $k\in \N$. 
For $\vec{i}\in V_{\G^{\odot k}}$ (that is, $\vec{i}$ is a nondecreasing element of $[n]^k$), 
\begin{equation*}
N_{\G^{\odot k}}(\vec{i})=\left \{\vec{j}\in V_{\G^{\odot k}}:\Orb(\vec{j})\cap \left ( N(v_1)_{\G}^{\vec{m}(\vec{i})_1}\times N(v_2)_{\G}^{\vec{m}(\vec{i})_2}\times \cdots \times N(v_n)_{\G}^{\vec{m}(\vec{i})_n}\right )\neq \emptyset \right \}.
\end{equation*}

Moreover, its cardinality  is bounded by 
$$ \deg_{\G^{\odot k}}(\vec{i})\leq  \binom{ | \bigcup _{\vec{m}(\vec{x})_i>0}N_{\G}(v_i)  |+k-1}{k}\leq \prod_{\ell =1}^n\deg_{\G}(v_{\ell})^{\vec{m}(\vec{x})_{\ell}}.$$
\end{proposition}

\begin{proof}
 It follows from Proposition \ref{Lemma:MatrixRep} that the positivity of an entry requires the  existence of $(\vec{q},\vec{t})\in \Orb(\vec{x})\times \Orb(\vec{y})$ such that $\vec{t}_i\in N_{\G}(\vec{q}_{\ell})$ for every $\ell$. Therefore the elements must  belong to the union of the
 elements in $\vec{x}$. The second part follows from the first since there is a single  neighboring  set.
\end{proof}

If $\vec{i}=(v_{\ell},v_{\ell},\ldots, v_{\ell})$, with $v_{\ell}$ repeated $k$ times, then
$\deg_{\G^{\odot k}} (\vec{i}) =\binom{\deg_{\G}(v_\ell)+k-1}{k}$ since the neighbors of $\vec{i}$ are 
$k$-sequences of nondecreasing neighbors of $v_{\ell}$.
We can completely characterize the number of loops, edges, and the degrees of vertices in the second symmetric tensor power.

\begin{proposition}\label{Prop:Grados}
Let $\G$ be an undirected graph with vertex set $V_{\G} = \{v_1, v_2,\ldots ,v_n\}$. 
\begin{enumerate}[leftmargin=*]
    \item if $\vec{i}=(v_a,v_b)$, then 
    \begin{equation*}
        \deg_{\G^{\odot 2}}(\vec{i})=
        \begin{cases}
        \binom{\deg_{\G}(v_a)+1}{2} & \text{if $a=b$},\\
        \deg_{\G}(v_a)\deg_{\G}(v_b)-\binom{ |N_{\G}(v_a)\cap N_{\G}(v_b) |}{2} & \text{if $a<b$}.
        \end{cases}
    \end{equation*}
    
\item The number of loops on ${\G}^{\odot 2}$ is  $E_{\G}+\ell _d$, where
\begin{equation*}
\ell _d=\left| \big\{ \{v_1,v_2\}\not \in E_{\G}:\{v_1,v_1\},\{v_2,v_2\}\in E_{\G} \big\} \right|
\end{equation*}
is the number of non adjacent pairs of loops in $\G$.

\item The number of edges in $\G^{\odot 2}$ is 
{\footnotesize
\begin{equation*}
\frac{1}{2}\left (E_{\G}+\ell _d +\sum_{i\in [n]}\binom{\deg_{\G}(v_i)+1}{2}+\sum_{1\leq i<j\leq n}\deg _{\G}(v_i)\deg_{\G}(v_j)-\sum_{1\leq i<j\leq n}\binom{|N_{\G}(v_i)\cap N_{\G}(v_j)|}{2}\right )
\end{equation*}
}
\end{enumerate}
\end{proposition}

\begin{proof}
(a) If $a=b$, Proposition \ref{Prop:Degree} gives the result. If $a\neq b$, then every element in 
    $\deg _{\G^{\odot 2}}(\vec{i})$ belongs to $N(v_a)\times N(v_b)$. This overcounts the choices for which the chosen vertices can be in either position of the tuple. Then take out the number of ways to choose two elements in the intersection of the neighbours to obtain the result. 

\smallskip\noindent(b) There are two kinds of loops: those  for which $(a,b)$ goes to $(b,a)$ in $\G^{\odot 2}$ when $\{a,b\}\in E_{\G}$ or $(a,b)$ goes to $(a,b)$ when $a$ and $b$ have loops. This last case  is counted by $\ell _d$ when $\{a,b\}\not \in E_{\G}$.

\smallskip\noindent(c) The result is now  a consequence of the handshake lemma  adding over the degrees with the formula established above.  
\end{proof}

\begin{corollary}
If $\G$ is $1$-regular, then $\G^{\odot k}$ is $1$-regular.
\end{corollary}

\begin{proof}
Using the expression for the neighbourhood set, there is just one element that connects for $\vec{x}\in V_{\G^{\odot k}}$, mainly $\vec{y}$ such that the only element in $N(v_1)^{x_1}\times N(v_2)^{x_2}\times \cdots \times N(v_n)^{x_n}$ is contained in $\Orb(\vec{y})$.
\end{proof}

%

\begin{theorem}[Theorem 1 of \cite{weichsel-1962a}]\label{Theorem:Connectness}
(a) If $\G$ is connected and non-bipartite, then $\G^{\otimes 2}$ is connected. 
(b) If $\G$ is connected and bipartite, then $\G^{\otimes 2}$ has two connected components.
\end{theorem}

\begin{corollary}\label{Corollary:Conn}
(a) If $\G$ is connected and non-bipartite, then $\G^{\odot 2}$ is connected. 
(b) If $\G$ is connected and bipartite, then $\G^{\odot 2}$ has two components.
\end{corollary}

\begin{proof}
(a) This follows from Theorem \ref{Theorem:Connectness} and the graph-theoretic construction of the symmetric tensor product. 

\smallskip\noindent(b) If $(v_i,v_j)\in V_{\G^{\odot 2}}$, the connectivity of $\G$ provides a path from $v_i$ to $v_j$. 
This path lifts to a path in $\G^{\otimes 2}$ from $(v_i,v_j)$ to $(v_j,v_i)$, so identifying points in $\S_2$ orbits
does not decrease the number of connected components.
\end{proof}

\begin{corollary}
If $\G$ is a connected, non-bipartite, loopless graph, then $\G^{\odot k}$ is connected. 
In general, the number of connected components on $\G^{\odot k}$ is at most $2^{k-1}$.
\end{corollary}

\section{Properties of $\G^{\odot k}$ for particular graphs $\G$}\label{Section:Particular}

A number of curious features arise when one considers symmetric tensor powers of familiar graphs.
In this section we describe a variety of such phenomena.

\subsection{Complete graphs with loops}
The graph $\J_n$ is the graph on $n$ with all pairs of vertices adjacent; its adjacency matrix is the $n \times n$
all-ones matrix.  In particular, there is a loop at each vertex
and $\J_n$ has $n+\binom{n}{2}=\binom{n+1}{2}$ edges; see Figure \ref{Figure:Jn}.

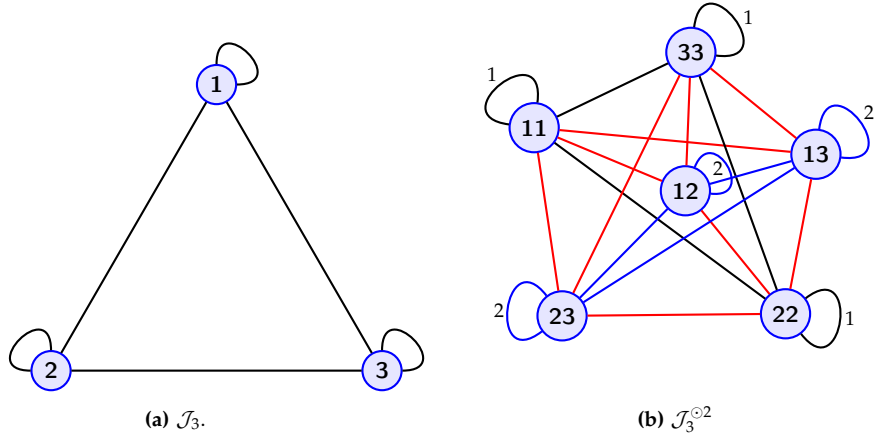
\begin{figure}[h]
\begin{subfigure}[b]{0.425\textwidth}
    \centering
    \begin{tikzpicture}[thick,scale=2.2, every node/.style={scale=0.75},auto, node distance=2cm,  thick,                           node_style/.style={circle,draw=blue,fill=blue!10!,font=\sffamily\Large\bfseries},                          edge_style/.style={draw=black,font=\small}]
\node[node_style] (v0) at (0,1.73) {1} ;
\node[node_style] (v1) at (-1,0) {2} ;
\node[node_style] (v2) at (1,0) {3} ;
\draw[edge_style]  (v0) -- (v1);
\draw[edge_style]  (v0) -- (v2);
\draw[edge_style]  (v1) -- (v2);
\draw (v0) to [out=0,in=90,looseness=5] (v0) ;
\draw (v1) to [out=180,in=90,looseness=5] (v1) ;
\draw (v2) to [out=0,in=90,looseness=5] (v2) ;
\end{tikzpicture} 
    \caption{$\J_3$.}
    \label{Figure:Jn}	        
    \end{subfigure}
\hfill
    \begin{subfigure}[b]{0.525\textwidth}	                
    \centering                        
    \begin{tikzpicture}[rotate=-10, thick,scale=0.75, every node/.style={scale=0.75},auto, node distance=2cm,  thick,                           node_style/.style={circle,draw=blue,fill=blue!10!,font=\sffamily\Large\bfseries},                          edge_style/.style={draw=black,font=\small}]
    \node[node_style] (v0) at (0.0,3.2012) {11} ;
    \node at (-0.9,4){$1$} ;
    \node[node_style] (v1) at (4.9609,0.7239) {22} ;
    \node at (6.1,0.85){$1$} ;
    \node[node_style] (v2) at (2.5035,5.0) {33} ;
    \node at (3.4,5.8){$1$} ;
    \node[node_style] (v3) at (2.8344,2.5648) {12} ;
    \node at (3.3344,3.0648){$2$} ;
    \node[node_style] (v4) at (5.0,3.6043) {13} ;
    \node at (5.8,4.5){$2$} ;
    \node[node_style] (v5) at (1.0666,0.0) {23} ;
    \node at (-0.05,-0.1){$2$} ;
    \draw[edge_style]  (v0) -- (v1);
    \draw[edge_style]  (v0) -- (v2);
    \draw[draw=red,font=\small]  (v0) -- (v3);
    \draw[draw=red,font=\small]  (v0) -- (v4);
    \draw[draw=red,font=\small]  (v0) -- (v5);
    \draw[edge_style]  (v1) -- (v2);
    \draw[draw=red,font=\small]  (v3) -- (v1);
    \draw[draw=red,font=\small]  (v4) -- (v1);
    \draw[draw=red,font=\small]  (v1) -- (v5);
    \draw[draw=red,font=\small]  (v3) -- (v2);
    \draw[draw=red,font=\small]  (v4) -- (v2);
    \draw[draw=red,font=\small]  (v5) -- (v2);
    \draw[draw=blue,font=\small]  (v3) -- (v4);
    \draw[draw=blue,font=\small]  (v3) -- (v5);
    \draw[draw=blue,font=\small]  (v4) -- (v5);
    \draw (v0) to [out=180,in=90,looseness=5] (v0) ;
    \draw (v1) to [out=45,in=-45,looseness=5] (v1) ;
    \draw (v2) to [out=0,in=90,looseness=5] (v2) ;
    \draw[blue] (v3) to [out=0,in=75,looseness=4] (v3) ;
    \draw[blue] (v4) to [out=0,in=90,looseness=5] (v4) ;
    \draw[blue] (v5) to [out=135,in=-135,looseness=5] (v5) ;
    \end{tikzpicture} \caption{$\J_3^{\odot 2}$} 
    \label{Figure:JGraphTb}	        
    \end{subfigure}
    \captionsetup{width=.95\linewidth}	        
    \caption{$\J_3$, the complete graph with loops on three vertices, and its second symmetric tensor power.
    In (b) {\textcolor{red}{red}} edges have weight $\sqrt{2}$, 
    {\textcolor{blue}{blue}} edges have weight $2$, and black edges have weight $1$. }                           
    \label{fig:JGraph}
    \end{figure}

\begin{proposition}
The weight of the edge connecting $\vec{i}$ and $\vec{j}$ in $\J_n^{\odot k}$ 
is $\sqrt{\binom{k}{\vec{m}(\vec{i})}\binom{k}{\vec{m}(\vec{j})}}$.
As unweighted graphs,  $\J_n^{\odot k}=\J_{\binom{n+k-1}{k}}$. 
\end{proposition}

\begin{proof}
    All vertices are connected, so this follows from Proposition \ref{Lemma:MatrixRep}.
\end{proof}

The weights in the loops of the second symmetric tensor power of $\J_n$ determine 
the size of the orbit of the corresponding vertex; see Figure \ref{Figure:JGraphTb}.
 
\begin{corollary}
    The number of edges in $\J_n^{\odot 2}$ is $\binom{ \binom{n+1}{2}+1}{2}$.
\end{corollary}

This agrees with Proposition \ref{Prop:Grados}.  The resulting sequence 
$0, 1, 6, 21, 55, 120, 231,\ldots$ is A002817 in OEIS.

\subsection{Path graphs}
The \emph{path graph} $\P_n$ is the graph with adjacency matrix 
$[a_{ij}]\in \M_n$ with $a_{ij} = 1$ if $|i-j|=1$ and $a_{ij} = 0$ otherwise;
see Figures \ref{Figure:Path1} and \ref{Figure:Path3}, which
suggest the connectivity of symmetric tensor powers of path graphs is of interest.

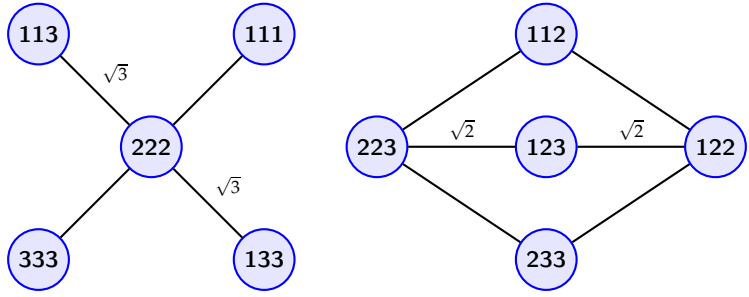
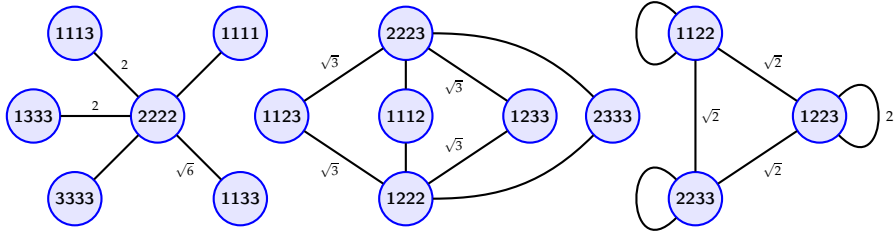
\begin{figure}[h]
\centering		
	        \begin{subfigure}[t]{\textwidth}
	                \centering
                        \begin{tikzpicture}[thick,scale=1.5, every node/.style={scale=0.75},auto, node distance=1.5,  thick,
                           node_style/.style={circle,draw=blue,fill=blue!10!,font=\sffamily\Large\bfseries},
                           edge_style/.style={draw=black,font=\small}]
                            \node[node_style] (v0) at (2,4) {111} ;
                            \node[node_style] (v1) at (1,3) {222} ;
                            \node[node_style] (v2) at (0,2) {333} ;
                            \node[node_style] (v3) at (4.5,4) {112} ;
                            \node[node_style] (v4) at (0,4) {113} ;
                            \node[node_style] (v5) at (6,3) {122} ;
                            \node[node_style] (v6) at (4.5,3) {123} ;
                            \node[node_style] (v7) at (2,2) {133} ;
                            \node[node_style] (v8) at (3,3) {223} ;
                            \node[node_style] (v9) at (4.5,2) {233} ;
                            \draw[edge_style]  (v0)--(v1);
                            \draw[edge_style]  (v1)--(v2);
                            \draw[edge_style]  (v4) edge node{$\sqrt{3}$} (v1);
                            \draw[edge_style]  (v1) edge node{$\sqrt{3}$} (v7);
                            \draw[edge_style]  (v3)--(v5);
                            \draw[edge_style]  (v3)--(v8);
                            \draw[edge_style]  (v6) edge node{$\sqrt{2}$} (v5);
                            \draw[edge_style]  (v5)--(v9);
                            \draw[edge_style]  (v8) edge node{$\sqrt{2}$} (v6);
                            \draw[edge_style]  (v8)--(v9);                        
                            \end{tikzpicture}
                        \caption{$\P_3^{\odot 3}$ has $\lceil \frac{3+1}{2} \rceil = 2$ connected components and no loops,
                        as predicted by Theorem \ref{Theorem:PathGraph}.b}        
                        \label{fig:PathGraph3b}      
	        \end{subfigure}	        
     
	        \begin{subfigure}[t]{1.0\textwidth}
	                \centering
                        \begin{tikzpicture}[thick,scale=0.55, every node/.style={scale=0.55},auto, thick,
                           node_style/.style={circle,draw=blue,fill=blue!10!,font=\sffamily\Large\bfseries},
                           edge_style/.style={draw=black,font=\small}]                        
                            \node[node_style] (v0) at (-4,2) {1111} ;
                            \node[node_style] (v1) at (-6,0) {2222} ;
                            \node[node_style] (v2) at (-8,-2) {3333} ;
                            \node[node_style] (v3) at (0,0) {1112} ;
                            \node[node_style] (v4) at (-8,2) {1113} ;
                            \node[node_style] (v5) at (7,2) {1122} ;
                            \node[node_style] (v6) at (-3,0) {1123} ;
                            \node[node_style] (v7) at (-4,-2) {1133} ;
                            \node[node_style] (v8) at (0,-2) {1222} ;
                            \node[node_style] (v9) at (10,0) {1223} ;
                            \node at (11.7,0){$2$};
                            \node[node_style] (v10) at (3,0) {1233} ;
                            \node[node_style] (v11) at (-9,0) {1333} ;
                            \node[node_style] (v12) at (0,2) {2223} ;
                            \node[node_style] (v13) at (7,-2) {2233} ;
                            \node[node_style] (v14) at (5,0) {2333} ;
                            \draw[edge_style]  (v0)--(v1);
                            \draw[edge_style]  (v1)--(v2);
                            \draw[edge_style]  (v4) edge node{$2$} (v1);
                            \draw[edge_style]  (v7) edge node{$\sqrt{6}$} (v1);
                            \draw[edge_style]  (v11) edge node{$2$} (v1);
                            \draw[edge_style]  (v3)--(v8);
                            \draw[edge_style]  (v3)--(v12);
                            \draw[edge_style]  (v5) edge node{$\sqrt{2}$} (v9);
                            \draw[edge_style]  (v5) edge node{$\sqrt{2}$} (v13);
                            \draw[edge_style]  (v8) edge node{$\sqrt{3}$} (v6);
                            \draw[edge_style]  (v6) edge node{$\sqrt{3}$} (v12);
                            \draw[edge_style]  (v8) edge node{$\sqrt{3}$} (v10);
                            \draw[edge_style]  (v8)  to [out=0, in=-135, looseness=1] (v14);
                            \draw[edge_style]  (v9) edge node{$\sqrt{2}$} (v13);
                            \draw[edge_style]  (v10) edge node{$\sqrt{3}$} (v12);
                            \draw[edge_style]  (v12) to [out=0, in=135, looseness=1] (v14);
                            \draw (v5) to [out=-135,in=135,looseness=5] (v5) ;
                            \draw (v9) to [out=45,in=-45,looseness=5] (v9) ;
                            \draw (v13) to [out=-135,in=135,looseness=5] (v13) ;   
                            \end{tikzpicture}
                        \caption{$\P_3^{\odot 4}$ has $\lceil \frac{4+1}{2} \rceil = 3$ connected components
                        and $\binom{3 + 2-2}{2} = 3$ loops, as predicted by Theorem \ref{Theorem:PathGraph}.a.}        
                        \label{fig:PathGraph4b}      
	        \end{subfigure}
	            \captionsetup{width=.95\linewidth}	        
                        \caption{Symmetric tensor powers $\P_3^{\odot 3}$ and $\P_3^{\odot 4}$
                        of the path graph $\P_3$ on three vertices. 
                        Only nonzero edge weights unequal to $1$ are included.}              
\label{Figure:P3}
\end{figure}

\begin{theorem}\label{Theorem:PathGraph}
    $\P_n^{\odot k}$ has $\lceil \frac{k+1}{2} \rceil$ connected components. 
    (a) If $k=2\ell$, then one of those components contains $\binom{n+\ell-2}{\ell}$ loops.
    (b) If $k$ is odd, then $\P_n^{\odot k}$ contains no loops.
\end{theorem}

\begin{proof}
If $\vec{i}$ and $\vec{j}$ share an edge in $\P_n^{\odot k}$, then
\begin{align*}
    \vec{m}(\vec{i})_{\text{odd}}=\sum _{\ell \text{ odd}}\vec{m}(\vec{i})_{\ell}&=\sum _{\ell \text{ even}}\vec{m}(\vec{j})_{\ell}=\vec{m}(\vec{j})_{\text{even}},\\
    \vec{m}(\vec{i})_{\text{even}}=\sum _{\ell \text{ even}}\vec{m}(\vec{i})_{\ell}&=\sum _{\ell \text{ odd}}\vec{m}(\vec{j})_{\ell}=\vec{m}(\vec{j})_{\text{odd}},
\end{align*}
because for every $r\in [n]$ $\vec{m}(\vec{i})_r$ is splitted and added to $\vec{m}(\vec{j})_{r+1}$ and $\vec{m}(\vec{j})_{r-1}$ if possible. 

Consider vertices $\vec{i}_{a,b}$ with $a+b=k$ and $a\geq b$ such that $\vec{m}(\vec{i}_{a,b})=(a,b,0,\cdots,0)$. All of them lie in 
different connected components, 
 otherwise if $\vec{i}_{a_1,b_1}$ and $\vec{i}_{a_2,b_2}$ with $a_1>a_2$ are in the same component, then $a_1=a_2$ or $a_1=b_2$ which is a contradiction.
 
 Furthermore, every vertex $\vec{i}=(v_1,v_2,\cdots ,v_k)$ is in the same connected component with a vertex of the form $\vec{i}_{a,b}$. To see
  this, consider moving from $\vec{m}(\vec{i})=(v_1,v_2,\cdots ,v_n)$ to $\vec{m}(\vec{i}')=(v_2,v_1+v_3,v_4,\cdots ,v_n,0)$ until the only nonzero entries are the first two, if needed swap the first two coordinates. This implies that the number of connected components is the same as the number of different vertices of the form $\vec{i}_{a,b}$.  There are $ \lceil \frac{k+1}{2}  \rceil$ such vertices, because there are $k+1$ ways to choose $a,b$ such that $a+b=k$ and $b \leq a$ gives the correct count.
  
There are no loops for $k$ odd. Otherwise a loop in $\vec{j}$ implies  $\vec{m}(\vec{j})_{\text{even}} = \vec{m}(\vec{j})_{\text{odd}}$. This 
contradicts  the parity of $k$. For $k$ even, write $k=2\ell$ and let  $a_{n,\ell}$ the number of loops in $\P_n^{\odot k}$. They 
 satisfy the recursion  
\begin{equation*}
a_{n,\ell}=a_{n-1,\ell}+a_{n,\ell-1},
\end{equation*}
with  initial conditions $a_{1,\ell}=0$ and $a_{n,1}=n-1$. The 
 argument is by induction. Note that  if $\vec{i}$ contains a loop, then either $\vec{m}(\vec{i})_n=0$ or not: if it is $0$ remove it, getting a vertex in $\P_{n-1}^{\odot k}$ that contain a loop. If $\vec{m}(\vec{i})_n\neq 0$, notice that $\vec{m}(\vec{i})_{n-1}\neq 0$ because otherwise it was not a loop and so consider $\vec{x'}$ given by $\vec{m}(\vec{x'})=(\vec{m}(\vec{i})_1,\cdots ,\vec{m}(\vec{i})_{n-1}-1,\vec{m}(\vec{i})_{n}-1)$, this is a loop in $\P_{n}^{\odot (k-2)}$. This implies the result by the usual recursions for binomial coefficients. 
\end{proof}

Figures \ref{Figure:P3} and \ref{Figure:P6} illustrate the previous theorem.

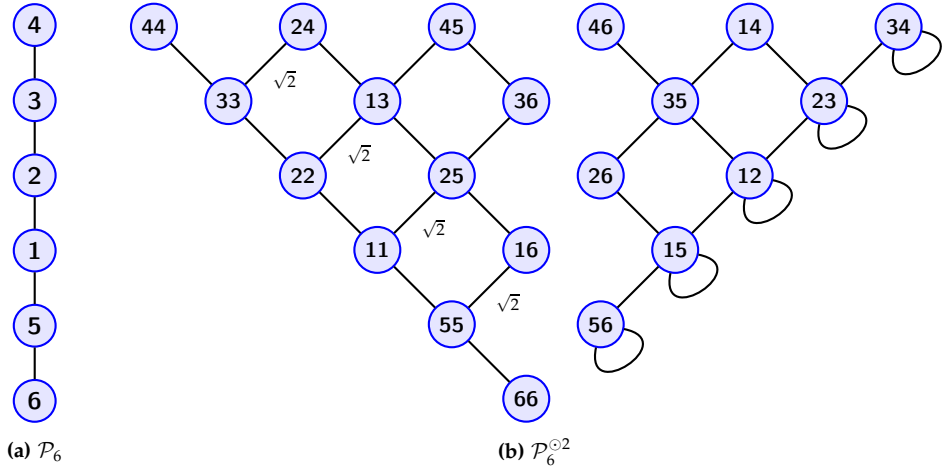
\begin{figure}
\centering    
\begin{subfigure}[t]{0.1\textwidth}
\centering                        
    \begin{tikzpicture}[thick,scale=0.5, every node/.style={scale=0.8},auto, node distance=2cm,  thick,                           node_style/.style={circle,draw=blue,fill=blue!10!,font=\sffamily\Large\bfseries},                          edge_style/.style={draw=black,font=\small}]
    \node[node_style] (v0) at (0,0) {1} ;
    \node[node_style] (v1) at (0,2) {2} ;
    \node[node_style] (v2) at (0,4) {3} ;
    \node[node_style] (v3) at (0,6) {4} ;
    \node[node_style] (v4) at (0,-2) {5} ;
    \node[node_style] (v5) at (0,-4) {6} ;
    \draw[edge_style]  (v0)--(v1);
    \draw[edge_style]  (v0)--(v4);
    \draw[edge_style]  (v1)--(v2);
    \draw[edge_style]  (v2)--(v3);
    \draw[edge_style]  (v4)--(v5);
\end{tikzpicture} 
\caption{$\P_6$}                         
\label{Figure:P6-1}	        
\end{subfigure}
\hfill
\begin{subfigure}[t]{0.85\textwidth}	                
\centering                        
\begin{tikzpicture}[thick,rotate=45,scale=0.7, every node/.style={scale=0.7},auto, node distance=2cm,  thick,                           node_style/.style={circle,draw=blue,fill=blue!10!,font=\sffamily\Large\bfseries},                          edge_style/.style={draw=black,font=\small}]
\node[node_style] (v0) at (0,0) {11} ;
\node[node_style] (v1) at (0,2) {22} ;
\node[node_style] (v2) at (0,4) {33} ;
\node[node_style] (v3) at (0,6) {44} ;
\node[node_style] (v4) at (0,-2) {55} ;
\node[node_style] (v5) at (0,-4) {66} ;
\node[node_style] (v6) at (6,-4) {12} ;
\node[node_style] (v7) at (2,2) {13} ;
\node[node_style] (v8) at (8,-2) {14} ;
\node[node_style] (v9) at (4,-4) {15} ;
\node[node_style] (v10) at (2,-2) {16} ;
\node[node_style] (v11) at (8,-4) {23} ;
\node[node_style] (v12) at (2,4) {24} ;
\node[node_style] (v13) at (2,0) {25} ;
\node[node_style] (v14) at (4,-2) {26} ;
\node[node_style] (v15) at (10,-4) {34} ;
\node[node_style] (v16) at (6,-2) {35} ;
\node[node_style] (v17) at (4,0) {36} ;
\node[node_style] (v18) at (4,2) {45} ;
\node[node_style] (v19) at (6,0) {46} ;
\node[node_style] (v20) at (2,-4) {56} ;
\draw[edge_style]  (v0)--(v1);
\draw[edge_style]  (v0)--(v4);
\draw[edge_style]  (v13) edge node{$ \sqrt{2} $} (v0);
\draw[edge_style]  (v1)--(v2);
\draw[edge_style]  (v7) edge node{$ \sqrt{2} $} (v1);
\draw[edge_style]  (v2)--(v3);
\draw[edge_style]  (v12) edge node{$ \sqrt{2} $} (v2);
\draw[edge_style]  (v4)--(v5);
\draw[edge_style]  (v10) edge node{$ \sqrt{2} $} (v4);
\draw[edge_style]  (v6)--(v16);
\draw[edge_style]  (v6)--(v9);
\draw[edge_style]  (v6)--(v11);
\draw[edge_style]  (v7)--(v18);
\draw[edge_style]  (v7)--(v12);
\draw[edge_style]  (v7)--(v13);
\draw[edge_style]  (v8)--(v16);
\draw[edge_style]  (v8)--(v11);
\draw[edge_style]  (v9)--(v20);
\draw[edge_style]  (v9)--(v14);
\draw[edge_style]  (v10)--(v13);
\draw[edge_style]  (v11)--(v15);
\draw[edge_style]  (v13)--(v17);
\draw[edge_style]  (v14)--(v16);
\draw[edge_style]  (v16)--(v19);
\draw[edge_style]  (v17)--(v18);
\draw (v6) to [out=-135,in=-45,looseness=5] (v6) ;
\draw (v9) to [out=-135,in=-45,looseness=5] (v9) ;
\draw (v11) to [out=-135,in=-45,looseness=5] (v11) ;
\draw (v15) to [out=-135,in=-45,looseness=5] (v15) ;
\draw (v20) to [out=-135,in=-45,looseness=5] (v20) ;
\end{tikzpicture} 
\caption{$\P_6^{\odot 2}$}                         
\label{Figure:P6-2}	        
\end{subfigure}
    \captionsetup{width=.95\linewidth}	        
\caption{Path graph on six vertices and its second symmetric tensor power.
Only edge weights unequal to $1$ are included.}    
\label{Figure:P6}
\end{figure}

\subsection{Cycle graphs}
The \emph{cycle graph} $\C_n$ is the graph with adjacency matrix $[a_{ij}]\in \M_n$, in which
$a_{ij} = 1$ if $|i-j| \in \{1, n-1\}$ and $a_{ij}=0$ otherwise; see Figure \ref{Figure:CycleGraph}.

\begin{figure}[h]
    \begin{subfigure}[t]{0.475\textwidth}
	                \centering
                        \begin{tikzpicture}[thick,scale=1.0, every node/.style={scale=1},auto, node distance=2,  thick,
                           node_style/.style={circle,draw=blue,fill=blue!10!,font=\sffamily\Large\bfseries},
                           edge_style/.style={draw=black,font=\small}]
                        
                            \node[node_style] (v1) at (-2,0) {1};
                            \node[node_style] (v2) at (0,2) {2};
                            \node[node_style] (v3) at (2,0) {3};
                            \node[node_style] (v4) at (2,-2) {4};
                            \node[node_style] (v5) at (-2,-2) {5};

                            \draw[edge_style]  (v1) edge node{1} (v2);
                            \draw[edge_style]  (v2) edge node{1} (v3);
                            \draw[edge_style]  (v3) edge node{1} (v4);
                            \draw[edge_style]  (v4) edge node{1} (v5);
                          \draw[edge_style]  (v1) edge node{1} (v5);
                            
                            \end{tikzpicture}
                        \caption{$\C_5$.} 
                        \label{fig:CycleGrapha}
	        \end{subfigure}
	        \begin{subfigure}[t]{0.475\textwidth}
	                \centering
                        \begin{tikzpicture}[thick,scale=0.3, every node/.style={scale=0.5},auto, node distance=1.5,  thick,
                           node_style/.style={circle,draw=blue,fill=blue!10!,font=\sffamily\Large\bfseries},
                           edge_style/.style={draw=black,font=\small}]
                        
                            \node[node_style] (v1) at (-2,0) {11};
                            \node[node_style] (v2) at (0,2) {22};
                            \node[node_style] (v3) at (2,0) {33};
                            \node[node_style] (v4) at (2,-2) {44};
                            \node[node_style] (v5) at (-2,-2) {55};
                            \node[node_style] (v21) at (-4,0) {25};
                            \node[node_style] (v22) at (0,4) {13};
                            \node[node_style] (v23) at (4,0) {24};
                            \node[node_style] (v24) at (4,-4) {35};
                            \node[node_style] (v25) at (-4,-4) {14};
                                                        \node[node_style] (v31) at (-6,0) {34};
                            \node[node_style] (v32) at (0,6) {45};
                            \node[node_style] (v33) at (6,0) {15};
                            \node[node_style] (v34) at (6,-6) {12};
                            \node[node_style] (v35) at (-6,-6) {23};
                            
                           \draw[edge_style]  (v1) edge node{} (v21);
                           \draw[edge_style]  (v31) edge node{} (v21);
                           \draw[edge_style]  (v2) edge node{} (v22);
                           \draw[edge_style]  (v32) edge node{} (v22);
                           \draw[edge_style]  (v3) edge node{} (v23);
                           \draw[edge_style]  (v33) edge node{} (v23);
                           \draw[edge_style]  (v4) edge node{} (v24);
                           \draw[edge_style]  (v34) edge node{} (v24);
                           \draw[edge_style]  (v5) edge node{} (v25);
                           \draw[edge_style]  (v35) edge node{} (v25);
                            \draw[edge_style]  (v1) edge node{} (v2);
                            \draw[edge_style]  (v2) edge node{} (v3);
                            \draw[edge_style]  (v3) edge node{} (v4);
                            \draw[edge_style]  (v4) edge node{} (v5);
                          \draw[edge_style]  (v1) edge node{} (v5);
                          
                          \draw[edge_style]  (v21) edge node{} (v22);
                            \draw[edge_style]  (v22) edge node{} (v23);
                            \draw[edge_style]  (v23) edge node{} (v24);
                            \draw[edge_style]  (v24) edge node{} (v25);
                          \draw[edge_style]  (v21) edge node{} (v25);
                          \draw[edge_style]  (v31) edge node{} (v32);
                            \draw[edge_style]  (v32) edge node{} (v33);
                            \draw[edge_style]  (v33) edge node{} (v34);
                            \draw[edge_style]  (v34) edge node{} (v35);
                          \draw[edge_style]  (v31) edge node{} (v35);
                          \draw (v31) to [out=0+135,in=225,looseness=5] (v31) ;
                          \draw (v32) to [out=135,in=45,looseness=5] (v32) ;
                          \draw (v33) to [out=0+45,in=0-45,looseness=5] (v33) ;
                          \draw (v34) to [out=0,in=0-90,looseness=5] (v34) ;
                          \draw (v35) to [out=-90,in=-90-90,looseness=5] (v35) ;
                            
                            \end{tikzpicture}
                        \caption{$\C_5^{\odot 2}$.}        \label{fig:CycleGraphb}      
	        \end{subfigure}
    \captionsetup{width=.95\linewidth}	        
    \caption{The cycle graph $\C_5$ and its second symmetric tensor power.}
    \label{Figure:CycleGraph}
\end{figure}
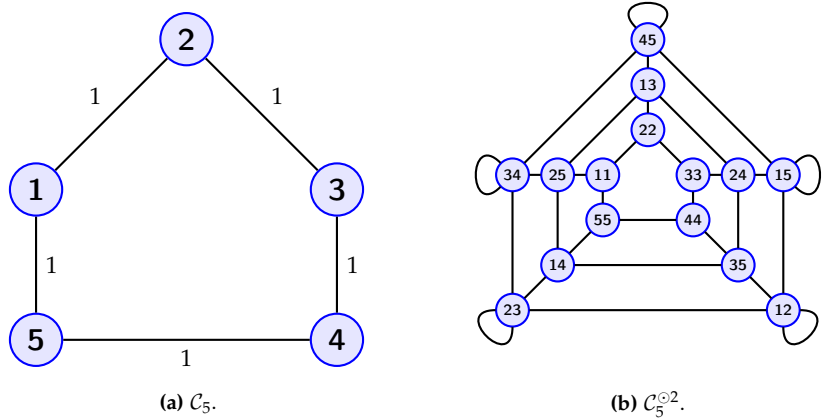

\begin{proposition}
(a) $\C_n^{\odot k}$ has $ \lceil \frac{k+1}{2} \rceil$ connected components if $n$ is even.
(b) $\C_n^{\odot k}$ is connected if $n$ is odd.
\end{proposition}

\begin{proof}
    (a) The case in which $n$ is even can be done in a similar fashion as in the proof of Proposition \ref{Theorem:PathGraph}.
    
    \smallskip\noindent(b)
     If $n$ is odd and $a_1>a_2$, there is a path 
    from $\vec{i}_{a_1,b_1}$ to $\vec{i}_{a_2,b_2}$, so $\C_{n}^{\odot k}$ is connected:
    {\small
    \begin{align*}
        \vec{m}(\vec{i}_{a_1,b_1})
        &=(a_1,b_1,\underbrace{0\cdots,0}_{\text{odd}})
        \rightarrow 
        (b_1,0,\underbrace{0\cdots,0}_{\text{even}},a_1)
        \rightarrow 
        (a_2,b_1,\underbrace{0\cdots,0}_{\text{odd}},a_2-a_1,0)\\
        &\rightarrow 
        (0,0,b_1,\underbrace{0\cdots,0}_{\text{odd}},a_2-a_1,0,a_2)
        \rightarrow 
        (a_2,b_1,0,\underbrace{0\cdots,0}_{\text{even}},a_2-a_1,0,0,0)\\
        &\rightarrow\cdots\rightarrow
        (0,0,b_1,0,a_2-a_1,0,\cdots ,0,a_2)
        \rightarrow 
        (a_2,b_1,0,a_2-a_1,0,\cdots,0)\\
        &\rightarrow 
        (0,0,b_1+a_2-a_1,0,\cdots,a_2)
        \rightarrow (a_2,b_1+a_2-a_1,0,\cdots,0)=\vec{m}(\vec{i}_{a_2,b_2}). \qedhere
    \end{align*}
    }
\end{proof}

\subsection{Complete bipartite graphs}
The \emph{complete bipartite graph} $\K_{n,m}$ has vertex set $[n+m]$ and edges connecting every vertex in $[n]$ with 
every vertex  in $[n+m]\setminus[n]$.
A vertex in $\K_{n,m}^{\odot k}$ is formed as follows: choose $0 \leq i \leq n$ and compose 
$i$ elements from $[n]$ and $k-i$ elements from $[n+m]\setminus [n]$. Observe that if $i<k-i$ this connection is possible by taking 
 $i$ elements from $[n+m]\setminus [n]$ and $k-i$ elements from $[n]$. For every $2i<k$, this forms a new bipartite graph
\begin{equation*}
\K_{\binom{i+n-1}{i}\binom{k-i+m-1}{k-i},\binom{i+m-1}{i}\binom{k-i+n-1}{k-i}}.
\end{equation*}
For $k$ even, write $2i=k$. Then
  any vertex having $i$ elements from each block is connected to every other 
 one, including  itself. This yields  a copy of the graph $\J_{\binom{k/2+n-1}{k/2}\binom{k/2+m-1}{k/2}}$. This proves the next result.
 
\begin{theorem}\label{Theorem:Bite}
For $k,n,m \in \N$, as unweighted graphs we have
\begin{align*}
    \K_{n,m}^{\odot k}&=\begin{cases}
    \displaystyle \bigcup _{i=0}^{\lfloor \frac{k-1}{2}\rfloor}
    \K_{\binom{i+n-1}{i}\binom{k-i+m-1}{k-i},\binom{i+m-1}{i}\binom{k-i+n-1}{k-i}} & \text{if $k$ is odd},\\
        \J_{\binom{k/2+n-1}{k/2}\binom{k/2+m-1}{k/2}} \cup \displaystyle \bigcup _{i=0}^{\lfloor \frac{k-1}{2}\rfloor} \K_{\binom{i+n-1}{i}\binom{k-i+m-1}{k-i},\binom{i+m-1}{i}\binom{k-i+n-1}{k-i}} & \text{otherwise}.
    \end{cases}
\end{align*}
This graph has $1+\lfloor \frac{k-1}{2}\rfloor +\frac{1 + (-1)^k}{2}=\lceil \frac{k+1}{2}\rceil$ connected components.
\end{theorem}

\begin{figure}
	                \centering
                      \begin{subfigure}[t]{0.45\textwidth}
                     \centering \begin{tikzpicture}[thick,scale=0.65, every node/.style={scale=0.75},auto, node distance=2.5,  thick,
                           node_style/.style={circle,draw=blue,fill=blue!10!,font=\sffamily\Large\bfseries},
                           edge_style/.style={draw=black,font=\small}]
                        
                            \node[node_style] (v1) at (-2,2) {111};
                            \node[node_style] (v2) at (-2,0) {112};
                            \node[node_style] (v3) at (-2,-2) {122};
                            \node[node_style] (v4) at (-2,-4) {222};
                        \node[node_style] (v5) at (0,-1) {333}; 
                        \node[node_style] (v6) at (2,2) {113};
                            \node[node_style] (v7) at (2,-1) {123};
                            \node[node_style] (v8) at (2,-4) {223};
                            \node[node_style] (v9) at (4,0) {133};
                        \node[node_style] (v10) at (4,-2) {233};
                        \draw[edge_style]  (v1) to (v5);
                            \draw[edge_style]  (v2) to (v5);
                            \draw[edge_style]  (v3) to (v5);
                            \draw[edge_style]  (v4) to (v5);
                            \draw[edge_style]  (v6) to (v9);
                      \draw[edge_style]  (v6) to (v10);
                      \draw[edge_style]  (v7) to (v9);
                      \draw[edge_style]  (v7) to (v10);
                      \draw[edge_style]  (v8) to (v9);
                      \draw[edge_style]  (v8) to (v10);
                            
                            \end{tikzpicture}
                            \caption{$\K_{2,1}^{\odot 3}$} 
\end{subfigure}
\quad
\begin{subfigure}[t]{0.45\textwidth}
                     \centering \begin{tikzpicture}[thick,scale=0.6, every node/.style={scale=0.5},auto, node distance=2.5,  thick,
                           node_style/.style={circle,draw=blue,fill=blue!10!,font=\sffamily\Large\bfseries},
                           edge_style/.style={draw=black,font=\small}]
                        
                            \node[node_style] (v1) at (-2,-1) {1111};
                            \node[node_style] (v2) at (-2,-2.5) {1112};
                            \node[node_style] (v3) at (-2,-4) {1122};
                            \node[node_style] (v4) at (-2,-5.5) {1222};
                    \node[node_style] (v11) at (-2,-7) {2222};
                        \node[node_style] (v5) at (0,-3) {3333}; 
                        \node[node_style] (v6) at (2,-1) {1113};
                            \node[node_style] (v7) at (2,-3) {1123};
                            \node[node_style] (v8) at (2,-5) {1223};
                            \node[node_style] (v12) at (2,-7) {2223};
                            \node[node_style] (v9) at (4,-2) {1333};
                            \node[node_style] (v13) at (6,-7) {1133};
                            \node[node_style] (v14) at (5,-5) {1233};
                            \node[node_style] (v15) at (7,-5) {2233};
                        \node[node_style] (v10) at (4,-3.5) {2333};
                        \draw[edge_style]  (v1) to (v5);
                            \draw[edge_style]  (v2) to (v5);
                            \draw[edge_style]  (v11) to (v5);
                            \draw[edge_style]  (v3) to (v5);
                            \draw[edge_style]  (v4) to (v5);
                            \draw[edge_style]  (v6) to (v9);
                      \draw[edge_style]  (v6) to (v10);
                      \draw[edge_style]  (v7) to (v9);
                      \draw[edge_style]  (v7) to (v10);
                      \draw[edge_style]  (v8) to (v9);
                      \draw[edge_style]  (v8) to (v10);
                       \draw[edge_style]  (v12) to (v9);
                      \draw[edge_style]  (v12) to (v10);
                      \draw[edge_style]  (v13) to (v14);
                        \draw[edge_style]  (v13) to (v15);
                        \draw[edge_style]  (v15) to (v14);
                           \draw (v13) to [out=0,in=0-90,looseness=3] (v13) ;
                           \draw (v14) to [out=15,in=0+90,looseness=3] (v14) ;
                           \draw (v15) to [out=15,in=0+90,looseness=3] (v15) ;
                           \end{tikzpicture}
                            \caption{$\K_{2,1}^{\odot 4}$} 
\end{subfigure}
                        \caption{The graphs $\K_{2,1}^{\odot 3}$ and $\K_{2,1}^{\odot 4}$ illustrate Theorem \ref{Theorem:Bite}.}              
\label{Figure:exKnm}
\end{figure}
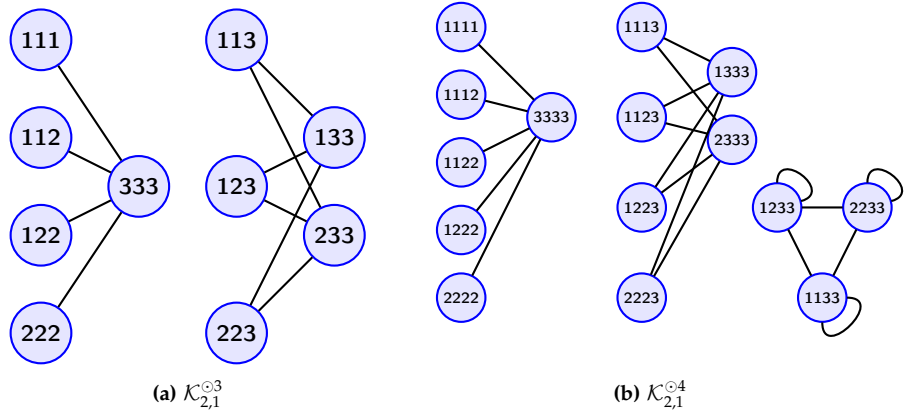

We illustrate the previous result as follows.
    Figure \ref{Figure:exKnm} presents the third and fourth symmetric power of $\K_{2,1}$;
    the second symmetric power is illustrated in Figure \ref{Figure:Compress}.
A complete bipartite graph is a \emph{star graph} if $n=1$; see Figure \ref{Figure:Star}.

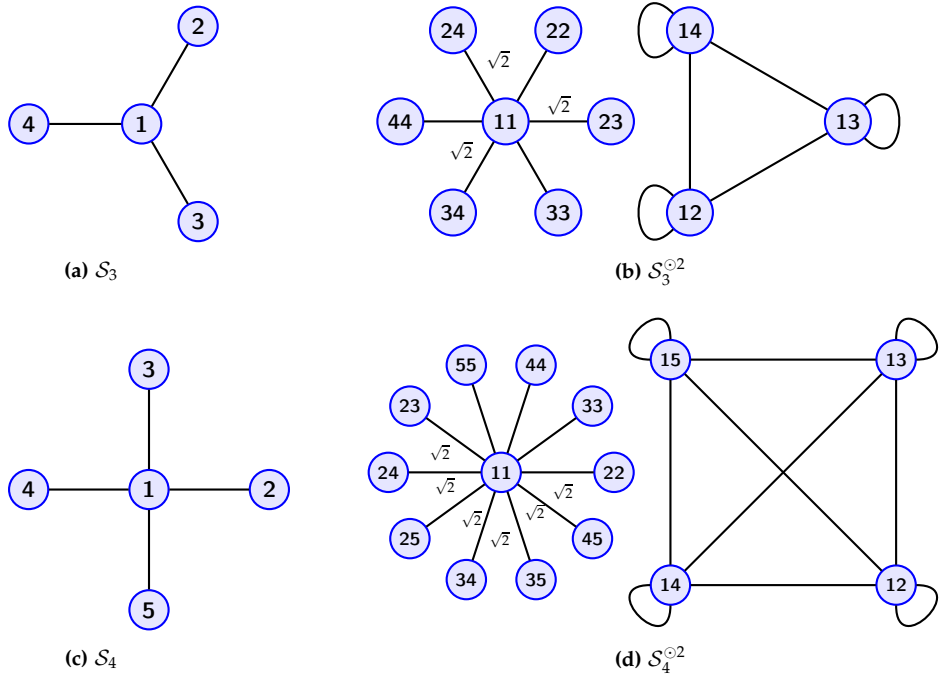
\begin{figure}  
\centering
\begin{subfigure}[t]{0.175\textwidth}	                
    \centering                        
        \begin{tikzpicture}[thick,scale=0.75, every node/.style={scale=0.75},thick,                           node_style/.style={circle,draw=blue,fill=blue!10!,font=\sffamily\Large\bfseries},                          edge_style/.style={draw=black,font=\small}]
        \node[node_style] (v0) at (0,0) {1} ;
        \node[node_style] (v1) at (1,1.73205) {2} ;
        \node[node_style] (v2) at (1,-1.73205) {3} ;
        \node[node_style] (v3) at (-2,0) {4} ;
        \draw[edge_style]  (v0)--(v1);
        \draw[edge_style]  (v0)--(v2);
        \draw[edge_style]  (v0)--(v3);
    \end{tikzpicture} 
    \caption{$\mathcal{S}_3$}                         
    \label{Figure:S3}	        
\end{subfigure}
\hfill
\begin{subfigure}[t]{0.65\textwidth}	                
\centering                        
\begin{tikzpicture}[thick,scale=0.7, every node/.style={scale=0.7},auto, node distance=2cm,  thick,                           node_style/.style={circle,draw=blue,fill=blue!10!,font=\sffamily\Large\bfseries},                          edge_style/.style={draw=black,font=\small}]
\clip (-2.5,-2.3) rectangle + (10.5,6);
\node[node_style] (v0) at (0,0) {11} ;
\node[node_style] (v1) at (1,1.73205) {22} ;
\node[node_style] (v2) at (1,-1.73205) {33} ;
\node[node_style] (v3) at (-2,0) {44} ;
\node[node_style] (v4) at (3.5,-1.73205) {12} ;
\node[node_style] (v5) at (6.5,0) {13} ;
\node[node_style] (v6) at (3.5,1.73205) {14} ;
\node[node_style] (v7) at (2,0) {23} ;
\node[node_style] (v8) at (-1,1.73205) {24} ;
\node[node_style] (v9) at (-1,-1.73205) {34} ;
\draw[edge_style]  (v0)--(v1);
\draw[edge_style]  (v0)--(v2);
\draw[edge_style]  (v0)--(v3);
\draw[edge_style]  (v0) edge node{$\sqrt{2}$} (v7);
\draw[edge_style]  (v8) edge node{$\sqrt{2}$} (v0);
\draw[edge_style]  (v9) edge node{$\sqrt{2}$} (v0);
\draw[edge_style]  (v4)--(v5);
\draw[edge_style]  (v4)--(v6);
\draw[edge_style]  (v5)--(v6);
\draw (v4) to [out=-135,in=135,looseness=5] (v4) ;
\draw (v5) to [out=45,in=-45,looseness=5] (v5) ;
\draw (v6) to [out=135,in=-135,looseness=5] (v6) ;
\end{tikzpicture} 
\caption{$\mathcal{S}_3^{\odot 2}$}
\label{Figure:S3x2}
\end{subfigure}
\\[10pt]
\begin{subfigure}[t]{0.175\textwidth}	                
\centering                        
\begin{tikzpicture}[thick,scale=0.8, every node/.style={scale=0.75},auto, node distance=2cm,  thick,                           node_style/.style={circle,draw=blue,fill=blue!10!,font=\sffamily\Large\bfseries},                          edge_style/.style={draw=black,font=\small}]
\node[node_style] (v0) at (0,0) {1} ;
\node[node_style] (v1) at (2,0) {2} ;
\node[node_style] (v2) at (0,2) {3} ;
\node[node_style] (v3) at (-2,0) {4} ;
\node[node_style] (v4) at (0,-2) {5} ;
\draw[edge_style]  (v0)--(v1);
\draw[edge_style]  (v0)--(v2);
\draw[edge_style]  (v0)--(v3);
\draw[edge_style]  (v0)--(v4);

\end{tikzpicture} 
\caption{$\mathcal{S}_4$}                         
\label{fig:CycleGraphTa}	        
\end{subfigure}
\hfill
\begin{subfigure}[t]{0.65\textwidth}	                
\centering                        
\begin{tikzpicture}[thick,scale=0.75, every node/.style={scale=0.6},auto, node distance=2cm,  thick,                           node_style/.style={circle,draw=blue,fill=blue!10!,font=\sffamily\Large\bfseries},                          edge_style/.style={draw=black,font=\small}]
\clip (-2.5,-2.8) rectangle + (10.3,5.6);
\node[node_style] (v0) at (0,0) {11} ;
\node[node_style] (v1) at (2,0) {22} ;
\node[node_style] (v2) at (1.61803, 1.17557) {33} ;
\node[node_style] (v3) at (0.618034, 1.90211) {44} ;
\node[node_style] (v4) at (-0.618034, 1.90211) {55} ;

\node[node_style] (v5) at (7,-2) {12} ;
\node[node_style] (v6) at (7,2) {13} ;
\node[node_style] (v7) at (3,-2) {14} ;
\node[node_style] (v8) at (3,2) {15} ;

\node[node_style] (v9) at (-1.61803, 1.17557) {23} ;
\node[node_style] (v10) at (-2,0) {24} ;
\node[node_style] (v11) at (-1.61803, -1.17557) {25} ;
\node[node_style] (v12) at (-0.618034, -1.90211) {34} ;
\node[node_style] (v13) at (0.618034, -1.90211) {35} ;
\node[node_style] (v14) at (1.61803, -1.17557) {45} ;
\draw[edge_style]  (v0)--(v1);
\draw[edge_style]  (v0)--(v2);
\draw[edge_style]  (v0)--(v3);
\draw[edge_style]  (v0)--(v4);
\draw[edge_style]  (v0) edge node{$ \sqrt{2} $} (v9);
\draw[edge_style]  (v0) edge node{$ \sqrt{2} $} (v10);
\draw[edge_style]  (v0) edge node{$ \sqrt{2} $} (v11);
\draw[edge_style]  (v0) edge node{$ \sqrt{2} $} (v12);
\draw[edge_style]  (v0) edge node{$ \sqrt{2} $} (v13);
\draw[edge_style]  (v0) edge node{$ \sqrt{2} $} (v14);
\draw[edge_style]  (v5)--(v6);
\draw[edge_style]  (v5)--(v7);
\draw[edge_style]  (v5)--(v8);
\draw[edge_style]  (v6)--(v7);
\draw[edge_style]  (v6)--(v8);
\draw[edge_style]  (v7)--(v8);
\draw (v5) to [out=0,in=-90,looseness=5] (v5) ;
\draw (v6) to [out=0,in=90,looseness=5] (v6) ;
\draw (v7) to [out=-90,in=180,looseness=5] (v7) ;
\draw (v8) to [out=90,in=180,looseness=5] (v8) ;
\end{tikzpicture} 
\caption{$\mathcal{S}_4^{\odot 2}$}                         
\label{Figure:S4-2}	        
\end{subfigure}
    \captionsetup{width=.95\linewidth}	        
\caption{Star graphs are bipartite graphs.  The images above illustrate Theorem \ref{Theorem:Bite}.}    
\label{Figure:Star}
\end{figure}

\section{Wiener index}\label{Section:Wiener}
Let $\G$ be an undirected graph with $V_{\G} = \{v_1,v_2,\ldots,v_n\}$. The \emph{Wiener index} 
\begin{equation*}
W(\G)=\sum_{\substack{\{v_i,v_j\}\subseteq V_{\G}\\ v_i\neq v_j}}d_{\G}(v_i,v_j)
\end{equation*}
measures the complexity of $\G$.  Here 
$d_{\G}(v_i,v_j)$ is the minimum distance between $v_i$ and $v_j$ in $\G$.  In this section we compute 
the Wiener index of the symmetric tensor product for several types of graphs.

\begin{proposition}
\begin{enumerate}[leftmargin=*]
\item $W (\J_n^{\odot k} )= \displaystyle \binom{\binom{k+n-1}{k}}{2}$.

\item $W (\K_n^{\odot k} )=\displaystyle\binom{\binom{k+n-1}{k}}{2}+\frac{n}{2}\Bigg (\binom{k+2(n-2)+1}{2(n-1)}-\binom{k-\left \lceil \frac{k}{2} \right \rceil-\left (\frac{1 + (-1)^k}{2}\right )+n-1}{n-1}\Bigg )$.
\end{enumerate}
\end{proposition}

\begin{proof}
It is clear that the graph $\J_n^{\odot k}$ has every possible edge and so, every vertex is at distance one of every other vertex, implying that 
the sum is given by choosing the two vertices. For the complete graph, every two vertices are at distance no more than two.
Call $D_{k}$ the number of pairs of vertices, without order, that are not connected by an edge. This happens only if there is one element 
appearing more times than the rest of the elements. In this situation, using the pigeonhole principle, we will have to assign one of the copies of this vertex to itself, but the complete graph contains no loops, hence it is impossible to go by an edge. Let $\{\vec{i},\vec{j}\}$ be one of such pairs, then there is a $r \in [n]$ such that $\vec{i}_s>\sum_{r\neq s}\vec{j}_r=k-\vec{j}_s$, and so the condition becomes $\vec{i}_s+\vec{j}_s>k$. Notice also that this can only happen in one index $s$, and without loss of generality, we can say that it happens in the first component of the compositions. An expression for the Wiener index of the complete graph becomes
\begin{eqnarray}
W\left (\K_n^{\odot k}\right )&=1\cdot \left (\displaystyle {\binom{k+n-1}{k}\choose 2}-n\cdot D_{k}\right )+2\cdot n\cdot D_{k}\\
&=\displaystyle {\binom{k+n-1}{k}\choose 2}+2\cdot n\cdot D_{k},
\end{eqnarray}
   \noindent 
but $T_{k}=2\cdot D_{k}+F_{k},$
where $T_k=\left |\{(\vec{i},\vec{j}):\vec{i}_1+\vec{j}_1>k\}\right |$ denotes the number of pairs of compositions of $k$ of size $n$ for which the sum of their first parts exceeds $k$ and $F_{k}=\left |\{\vec{i}:2\vec{i}_1>k\}\right |$ corresponds to those pairs for which the two compositions are the same. 
Using stars and bars one gets $F_{k}=\displaystyle \binom{k-\left \lceil \frac{k}{2} \right \rceil-(\frac{1 + (-1)^k}{2})+n-1}{n-1}$, where $[2|k]$ is $1$ if $k$ is even and $0$ otherwise.\\
On the other hand, 
\begin{align*}
    T_{k}&=\sum_{s=1}^k\sum_{t=k+1-s}^k\binom{k-s+n-2}{n-2}\binom{k-t+n-2}{n-2}\\
    &=\sum_{s=1}^k\sum_{t=0}^{s-1}\binom{k-s+n-2}{n-2}\binom{t+n-2}{n-2}\\
    &=\sum_{s=1}^k\binom{k-s+n-2}{n-2}\binom{s+n-2}{n-1}=\binom{k+2(n-2)+1}{2(n-2)}.
\end{align*}
\noindent
The proof is complete. 
\end{proof}

Recall that $\C_n$ is the cycle graph on $N$ vertices.
Corollary \ref{Corollary:Conn} shows that if $n$ is odd, then $\C_n^{\odot 2}$ is connected. 
We compute its Wiener index. 

\begin{proposition}
Let $n\geq 3$ be odd, then $\C_n^{\odot 2}$ is connected and
\begin{equation*}
W(\C_n^{\odot 2})=
\binom{n+2}{2}W(\C_n )+\left (n\binom{n}{2}-2W\left (C_n\right )\right )\binom{\frac{n+3}{2}}{2}-2n^2\binom{\frac{n+3}{2}}{3}.
\end{equation*}
\end{proposition}

\begin{proof}
The set of vertices of $\C_n^{\odot 2}$ can be identified with $\{ v_i \odot v_j: i \leq j\}$.
Partition this set into $\frac{n+1}{2}$ blocks of $n$ vertices such that the graph induced by each of the parts is isomorphic to $\C_n$. 
Consider the partition of the vertices of $\C_n^{\odot 2}$ given by 
\begin{equation*}
V_{\C_n^{\odot 2}}=\bigcup _{\ell = 0}^{(n-1)/2} V_{\ell},
\end{equation*}
in which $V_{\ell}=\{v_{i} \odot v_{j}:j=i+\ell\}$. Each block of this partition is connected to other two 
(except the blocks $v_i \odot v_i$ and $v_i \odot v_{i+1}$) by the edges $v_{i} \odot v_{j}$ to $v_{i+1} \odot v_{j-1}$ and  
$v_{i-1} \odot v_{j+1}$ as seen in Figure \ref{fig:CycleGraphb}. Thus, $\C_n^{\odot 2}$ is connected and its 
Wiener index depends on which of the $\frac{n+1}{2}$ blocks we are taking the vertices. If $d$ represents the distance from one block to the other,
then 
the distance in between the elements of each block is $d$. There are $n^2$ choices of vertices  and  $\frac{n+1}{2}-d$ choices for the two blocks with distance $d$, so 
\begin{equation*}
    W(\C_n^{\odot 2})
    =\sum_{i=0}^{\frac{n+1}{2}}\left (2W (\C_n)+i\cdot n^2\right )\left (\frac{n+1}{2}-i\right )-\frac{n+1}{2}W (\C_n ).
\end{equation*}
The result follows by  expanding the product and using binomial identities.
\end{proof}

\section{Open questions}\label{Section:Open}

Although one can consider symmetric tensor products of distinct graphs, we have restricted our attention to symmetric powers
to avoid the necessary notational hurdles.  Obviously, the investigation of symmetric tensor products of distinct graphs is wide open territory
that warrants exploration.


\begin{problem}
Develop analogous results and observations for symmetric tensor products of distinct graphs.
\end{problem}

\begin{problem}\label{Problem:Nesting}
Figure \ref{Figure:Scepter} suggests that $\G^{\odot k}$ is always contained in $\G^{\odot(k+1)}$
in the sense of weighted graphs (edge weights can increase as one moves to a higher power).  
It looks like one can prepend $1$ to each vertex label to pass to the next power.  But this fails in 
Figure \ref{Figure:Barbell}, even though nesting occurs.  What can be said about nesting in
successive symmetric tensor powers?
\end{problem}

\begin{problem}
If $k \leq r$, is $\G^{\odot k}$ a subgraph of $\G^{\odot r}$ in the sense of weighted graphs.
That is, are the weights in $\G^{\odot k}$ at most the corresponding weights in $\G^{\odot r}$?
\end{problem}

Both problems have a positive answer if one of the vertices has a loop. Just call that vertex $1$ and prepend it to the sequence. 
If not, as in the example of $\P_3$, one can think of picking an edge and alternating it.

\begin{problem}
In Appendix \ref{Section:WellDefined} we show how the symmetric tensor power of a permutation matrix is a permutation matrix.
How are the cycle decompositions of a permutation and its symmetric tensor powers related?
\end{problem}


\appendix
\section{The Symmetric tensor power of a graph is well-defined}\label{Section:WellDefined}

We now show that isomorphic graphs, that is, graphs whose adjacency matrices are permutation
similar, have isomorphic symmetric tensor powers.

\begin{theorem}\label{Theorem:WellDefined}
The symmetric tensor power operation on graphs is well defined; that is, isomorphic graphs have isomorphic symmetric powers.
\end{theorem}

\begin{lemma}\label{Lemma:Permutation}
Let $\Gamma _{\sigma} \in \M_n$ represent the permutation $\sigma \in \mathfrak{S}_n$, that is,
\begin{equation*}
    [\Gamma_{\sigma} ] _{i,j}
    =
    \begin{cases}
    1 & \text{if $j=\sigma (i)$},\\
    0 & \text{if otherwise},
    \end{cases}
\end{equation*}
and let $N = \binom{n+k-1}{k}$.
Then $\Gamma_{\sigma}^{\odot k} \in \M_N$ is the permutation matrix associated to 
$\sigma ^{\odot k}\in \mathfrak{S}_N$, in which $\sigma^{\odot k}(\vec{x})$ is the only nondecreasing element of 
$[n]^k$ in $\Orb (\sigma (\vec{x}))$. 
\end{lemma}

\begin{proof}
Proposition \ref{Lemma:MatrixRep} gives
\begin{equation*}
\big[\Gamma _{\sigma} ^{\odot k} \big]_{\vec{x},\vec{y}}=\frac{1}{\sqrt{\binom{k}{\vec{m}(\vec{x})}\binom{k}{\vec{m}(\vec{y})}}}
\sum_{\substack{\vec{t}\in \Orb(\vec{x})\\\vec{q}\in \Orb(\vec{y})}}\prod _{\ell = 1}^k\left [\Gamma_{\sigma}\right ] _{\vec{t}_{\ell},\vec{q}_{\ell}}.
\end{equation*}
If $\vec{z}\in \Orb(\sigma (\vec{x}))$ and $\vec{y} \neq \vec{z}$,  then $\left [\Gamma _{\sigma}\right ] _{\vec{t}_{\ell},\vec{q}_{\ell}}=0.$ 
This  happens  exactly $| \Orb(\vec{x})|=\binom{k}{\vec{m}(\vec{x})}$ times. 
Since $\binom{k}{\vec{m}(\vec{x})}=\binom{k}{\vec{m}(\sigma (\vec{x}))}$, 
it follows that $[\Gamma_{\sigma} ^{\odot k}  ]_{\vec{x},\sigma (\vec{x})}=1$. 
Thus, $\Gamma_{\sigma} ^{\odot k}$ is the permutation matrix corresponding to $\sigma ^{\odot k}$ in the group of permutations of the
$N$ nondecreasing elements of $[n]^k$ (we identify this group with $\mathfrak{S}_N$).
\end{proof}

\begin{example}
    Let $n=k=2$ and let $P = [\begin{smallmatrix} 0&1\\1&0 \end{smallmatrix}]$.  
    Then \eqref{eq:Ank22} and \eqref{eq:Ank23} ensure that
    \begin{equation*}
        P^{\odot 2}
        = \Big[
        \begin{smallmatrix}
        \0 & 1 &  \0 \\
        1 & \0 &  \0 \\
        \0 & \0 & 1
        \end{smallmatrix}
        \Big]
        \quad\text{and}\quad
        P^{\odot 3}
        = \bigg[
    \begin{smallmatrix}
    \0 & 1 & \0 & \0 \\
    1 & \0 & \0 & \0\\
    \0 & \0 & \0 & 1 \\
    \0 & \0 & 1 & \0
    \end{smallmatrix}
    \bigg]
    .    
    \end{equation*}
    Let $\vec{h}=(1,1),\vec{i}=(2,2)$ and $\vec{j}=(1,2)$ be the three posible indices when $n=k=2$. Consider $\sigma = 21$ a permutation on two elements, then $\sigma ^{\odot 2}(\vec{h})=\vec{i},$ $\sigma ^{\odot 2}(\vec{i})=\vec{h}$ and $\sigma ^{\odot 2}(\vec{j})=\vec{j}$. 
    Order the basis elements as $(1,1),(2,2),(1,2)$ in the nondecreasing elements of $[2]^2$. 
\end{example}

\begin{example}
    Let $n=3$, $k=2$, and let $P = \Big[ \begin{smallmatrix} \0 &1 &\0 \\ \0 & \0 & 1 \\ 1 &\0 & \0 \end{smallmatrix} \Big]$.  
    Then \eqref{eq:Ank32} ensures that
    \begin{equation*}
        P^{\odot 2}
        = \left[
    \begin{smallmatrix}
        \0 & 1 & \0 & \0 & \0 & \0 \\
        \0 & \0 & 1 & \0 & \0 & \0 \\
        1 & \0 & \0 & \0 & \0 & \0 \\
        \0 & \0 & \0 & \0 &  \0 & 1 \\
        \0 & \0 & \0 & 1 & \0 & \0 \\
        \0 & \0 & \0 & \0 & 1 & \0
    \end{smallmatrix}
    \right].
    \end{equation*}
 By imposing a linear order on the basis of the symmetric tensor product using the lexicographic, one can check that the list of permutations obtained for $n=3,k=2$ is
    $(123)^{\odot 2}=(123456)$,
    $(132)^{\odot 2}=(132546)$,
    $(213)^{\odot 2}= (213465)$,
    $(231)^{\odot 2}= (231645)$,
    $(312)^{\odot 2}= (312564)$, and
    $(321)^{\odot 2}= (321654)$.
    Theorem \ref{Theorem:Subgraph} ensures that each permutation is a prefix of its second symmetric power. 
\end{example}

\begin{proof}[Pf.~of Theorem \ref{Theorem:WellDefined}]
Suppose that $\G$ and $\mathcal{H}$ are isomorphic graphs with adjacency matrices $A$ and $B$, respectively.
Then there is a permutation $\sigma \in \mathfrak{S}_n$ and corresponding permutation matrix $\Gamma\in \M_n$ such that
$B = \Gamma A \Gamma^{\top}$.
By Lemma \ref{Lemma:Permutation} , $\Gamma^{\odot k}$ is the permutation matrix such that
 $[\Gamma_{\sigma}^{\odot k}]_{\vec{i},\vec{j}}=1$
 if and only if $\vec{j}\in \Orb(\sigma (\vec{i}))$.
For $A \in \M_N$, we have
$[\Gamma A\Gamma^{\top} ] _{i,j} = [A]_{\sigma (i),\sigma (j)}$.
Lemma \ref{Lemma:MatrixRep} shows that
\begin{align*}
   [ ( \Gamma A\Gamma^{\top} )^{\odot k} ]_{\vec{i},\vec{j}}
   &=\frac{1}{\sqrt{\binom{k}{\vec{m}(\vec{i})}\binom{k}{\vec{m}(\vec{j})}}}\sum _{\substack{\vec{t}\in \Orb(\vec{i})\\ \vec{q}\in \Orb(\vec{j})}}\prod _{\ell =1}^kA_{\sigma (\vec{t}_{\ell}),\sigma (\vec{q}_{\ell})}\\
  &=\frac{1}{\sqrt{\binom{k}{\vec{m}(\vec{i})}\binom{k}{\vec{m}(\vec{j})}}}\sum _{\substack{\vec{t}\in \Orb(\sigma (\vec{i}))\\ \vec{q}\in \Orb(\sigma (\vec{j}))}}\prod _{\ell =1}^kA_{\vec{t}_{\ell}, \vec{q}_{\ell}}\\
  &=\frac{1}{\sqrt{\binom{k}{\vec{m}(\sigma (\vec{i}))}\binom{k}{\vec{m}(\sigma (\vec{j}))}}}\sum _{\substack{\vec{t}\in \Orb(\sigma (\vec{i}))\\ \vec{q}\in \Orb(\sigma (\vec{j}))}}\prod _{\ell =1}^kA_{\vec{t}_{\ell}, \vec{q}_{\ell}}\\
  &=\big[A^{\odot k}\big] _{\sigma (\vec{i}),\sigma (\vec{j})}\\
  &=\big[\Gamma ^{\odot k}A^{\odot k} (\Gamma ^T ) ^{\odot k} \big] _{\vec{i},\vec{j}}.
\end{align*}

Thus, $A^{\odot k}$ and $B^{\odot k}$ are permutation similar and hence $\G^{\odot k}$ and $\mathcal{H}^{\odot k}$
are isomorphic graphs.
\end{proof}

\end{document}